\documentclass[psamsfonts]{amsart}

\usepackage{amssymb,amsfonts}
\usepackage[all,arc]{xy}
\usepackage{enumerate}
\usepackage{mathrsfs}
\usepackage{ytableau}
\usepackage{genyoungtabtikz}
\usepackage{subfig}

\theoremstyle{definition}
\newtheorem{thm}{Theorem}[section]
\newtheorem{cor}[thm]{Corollary}
\newtheorem{prop}[thm]{Proposition}
\newtheorem{lem}[thm]{Lemma}

\theoremstyle{definition}
\newtheorem{defn}[thm]{Definition}

\theoremstyle{remark}
\newtheorem{rem}[thm]{Remark}

\makeatletter
\let\c@equation\c@thm
\makeatother
\numberwithin{equation}{section}

\begin{document}

\title[Specht modules in the principal block]
{\bf On the structure of Specht modules \\ in the principal block of $F\Sigma_{3p}$}
 
\author{\sc Michael Rosas}
\address
{Department of Mathematics\\ University at Buffalo, SUNY \\
244 Mathematics Building\\Buffalo, NY~14260, USA}
\thanks{Research was supported in part by NSF
grant  DMS-1068783} \email{marosas@buffalo.edu}

\maketitle

\begin{abstract}
Let $F$ be a field of characteristic $p$ at least 5.  We study the Loewy structures of Specht modules in the principal block of $F\Sigma_{3p}$.  We show that a Specht module in the block has Loewy length at most 4 and composition length at most 14.  Furthermore, we classify which Specht modules have Loewy length 1, 2, 3, or 4, produce a Specht module having 14 composition factors, describe the second radical layer and the socle of the reducible Specht modules, and prove that if a Specht module corresponds to a partition that is $p$-regular and $p$-restricted then the head of the Specht module does not extend the socle.    
\end{abstract}

\section{Introduction}
In \cite{CT}, Chuang and Tan determine the Loewy structures of Specht modules in defect 2 over odd characteristic.  The authors prove results for the Schur algebra, and then translate to results for the symmetric group by applying the Schur functor.  The authors mention that the symmetric group results may be proven entirely in the representation theory of the symmetric group.  

Throughout this paper $\Sigma_n$ is the symmetric group on $n$ letters and $F$ is a field of characteristic $p$ at least 5.  The goal of this paper is to establish some Loewy structure results for Specht modules which lie in the principal block of $F\Sigma_{3p}$.  For the remainder of the paper, the principal block of $F\Sigma_{3p}$ will be denoted by $\mathcal{B}$.       



Weight $w \leq 2$ Specht modules are well understood.  When $w = 0$ the Specht modules are simple and projective, and each $p$-block contains exactly one Specht module.  For $w = 1$ the Ext$^1$-quiver is a line segment, and the Specht modules have Loewy length and composition length at most 2.  Over odd characteristic, if $w  = 2$ then a Specht module has Loewy length at most 3 and composition length at most 5 (cf.  Proposition 6.2 of \cite{CT}  and Theorem 5.6 of \cite{DE}).       

In this paper we prove four structure results for Specht modules which lie in the principal block of $F\Sigma_{3p}$.  The first result shows that Specht modules in the block have Loewy lengths bounded above by 4.  The second result is a necessary and sufficient condition on a partition $\lambda$ so that the corresponding Specht module (in the principal block) attains the maximum possible Loewy length in Corollary \ref{cor:nscondition}.  The third result shows that if $S^{\lambda}$ is a reducible Specht module in the principal block with $\lambda$ $p$-regular then its second radical layer is determined by the Ext$^1$-quiver of the block.  For such a partition $\lambda$ we also prove that if $S^{\lambda}$ has Loewy length 3 then $\text{rad}^2(S^{\lambda}) = \text{soc}(S^{\lambda})$.  Finally, the fourth result gives an (attained) upper bound of 14 on the composition length for Specht modules in the block.  The approach taken to prove the mentioned results relates partitions in the principal block of $F\Sigma_{3p}$ to partitions in weight 2 blocks of $F\Sigma_{3p-1}$ then inducing back up to the principal block.  This approach is motivated by the work of Martin and Russell 
in \cite{MR}, where the authors show that a substantial portion of the $\text{Ext}^1$-quiver for the principal block of $F\Sigma_{3p}$ is built from $\text{Ext}^1$-quivers of the defect 2 blocks of $F\Sigma_{3p-1}$.    
    
For the case $p=3$, the Ext$^1$-quiver of the principal block, $B_0 = B_0(F\Sigma_9)$, is constructed in \cite{Tan2}.  The author gives the Loewy structures of the projective indecomposable modules of $B_0$ (cf. Theorems 2.15, 2.17 of \cite{Tan2}).  Since every projective indecomposable module of the symmetric group is a Young module one is able to deduce the following: If $\lambda$ is a partition in $B_0$ that is both $3$-regular and $3$-restricted then $S^{\lambda}$ has Loewy length 4.  Now if $\lambda$ is a partition in $B_0$ that is $3$-regular and not $3$-restricted then one may apply Lemma 2.14 of \cite{Tan2} and the arguments of Lemmas 
\ref{lem:sinduce}, \ref{lem:nicerest}, \ref{lem:getit}, and Theorem \ref{thm:gotit} of this paper to conclude $S^{\lambda}$ has Loewy length at most 3.  So the Specht modules in $B_0$ have Loewy length at most 4, and a Specht module having Loewy length 4 necessarily corresponds to a partition that is both 3-regular and 3-restricted.     
The representations in the blocks $B_0$ and $\mathcal{B}$ appear to behave similarly even though $B_0$ has non-abelian defect group and $\mathcal{B}$ has abelian defect group.  In fact, after analyzing the Ext$^1$-quiver of $B_0$, we see Proposition \ref{prop:mullextend} and Corollary \ref{cor:KS} of this paper are also true for $B_0$.                

In Section \ref{section:prelim} we recall some background material in the representation theory of the symmetric group.  In Section \ref{section:partitions} we prove a few results for certain weight 3 partitions of $n = 3p$.  In Section \ref{section:irreducibles} we classify all Specht modules lying in certain defect 2 blocks of $\Sigma_{3p-1}$ that are irreducible.  We then determine the Specht modules in $
\mathcal{B}$ that have an irreducible restriction to a defect 2 block of $F\Sigma_{3p-1}$.  Section \ref{section:loewy} contains the two main results of the paper, Corollaries \ref{cor:nscondition} and \ref{cor:KS}.  In addition, we see that the composition factors of the heart of the Specht module $S^{\lambda}$ are determined by the Ext$^1$-quiver of $\mathcal{B}$ whenever $\lambda$ is both $p$-regular and $p$-restricted (with help from the Mullineux map).  In Section \ref{section:bound} we prove that the largest attainable composition length for any Specht module in the principal block of $F\Sigma_{3p}$ is 14.  Finally, in Section \ref{section:conclude} we summarize the results of this paper and discuss future research for Specht modules in arbitrary blocks of weight 3.

\section{Preliminaries} \label{section:prelim}
A partition of $n$ is a sequence of non-negative integers $\lambda  = (\lambda_1, \ldots, \lambda_k)$ such that $\lambda_i \geq \lambda_{i+1}$ for all $i$ and $|\lambda| := \sum_{i} \lambda_i = n$. For brevity, write $\lambda \vdash n$ if $\lambda$ is a partition of $n$.  Assume $\lambda$ and $\mu$ are partitions of $n$.  Say $\lambda$ \emph{dominates} $\mu$, and write $\lambda \unrhd \mu$, if 
  $\sum_{i=1}^{\ell} \lambda_i \geq \sum_{i=1}^{\ell} \mu_i$ for all $\ell \geq 1$.  If $\lambda \unrhd \mu$ and $\lambda \neq \mu$, we say $\lambda$ \emph{strictly dominates} $\mu$ and write $\lambda \rhd \mu$.  This gives a partial order on the partitions of $n$ called the \emph{dominance order}.  Finally, let $\geq $ denote the lexicographic order on partitions.  
  
  \subsection{Young diagrams and $p$-regular partitions.}
Let $\lambda = (\lambda_1 ,\ldots, \lambda_k)$ be a partition of $n$.  The \emph{Young diagram} for $\lambda$ is the subset of $\mathbb{N} \times \mathbb{N}$:
$$[\lambda] = \{(i,j) \in \mathbb{N} \times \mathbb{N}:  1\leq  i \leq k, \ 1\leq j \leq \lambda_i  \}.$$  
There is a natural correspondence between partitions of $n$ and Young diagrams with $n$ boxes.  We identify a partition $\lambda$ with its Young diagram and write $\lambda$ to mean $[\lambda]$.  Let $\lambda'$ denote the conjugate partition to $\lambda$, that is, the partition with Young diagram:
$$[\lambda ' ] =  \{(j,i) \in \mathbb{N} \times \mathbb{N}:  1\leq  i \leq k, \ 1\leq j \leq \lambda_i  \}.$$    
\noindent In other words, $\lambda' = (\lambda_1' , \lambda_2', \cdots, \lambda_m')$ where $\lambda_j'$ is the number of boxes in the $j$th column of the Young diagram for $\lambda$.  
A node $A = (i,j)$ of $\lambda$ is called a \emph{removable} node for $\lambda$ if $\lambda - \{A\}$ is a partition.  In this case define $\lambda_A := \lambda - \{A\}$.  A node $B = (i,j) \in \mathbb{N}\times \mathbb{N}$ is called an \emph{addable} node for $\lambda$ if $\lambda \cup \{B\}$ is a partition.  In this case define $\lambda^B := \lambda \cup \{B\}$.  If $A = (i_1,j_1)$ and $B = (i_2, j_2)$ are nodes we say $B$ is \emph{above} $A$ or  $A$ is \emph{below} $B$ if $i_2 < i_1$.  

Fix a prime $p$.  Define the $p$-residue of a node $A = (i,j) $ in $\mathbb{N} \times \mathbb{N}$ to be $(j - i) \mod p$.  We denote the $p$-residue of the node $A$ by res$A$.  We refer to a node of $p$-residue $i$ as an $i$-node.  A removable $i$-node $A$ is \emph{normal} if there is an injective function from the set of addable $i$-nodes above $A$ to the set of removable $i$-nodes above $A$ which maps each addable node $B$ to a node $C$ below $B$.  Finally, a removable $i$-node is called \emph{good} if it is the lowest normal $i$-node.

\vskip .1in
\noindent \textbf{Example.}
Let $\lambda = (6,4,2^2,1^2)$ and $p=5$.  The removable nodes of $\lambda$ are $A_1 = (1, 6)$, $A_2 = (2,4)$, $A_3 = (4,2)$, and $A_4 = (6,1)$, and we fill the box at node $A_i$ in the Young diagram with res$A_i$.  We also mark each addable node of $\lambda$ with its 5-residue.

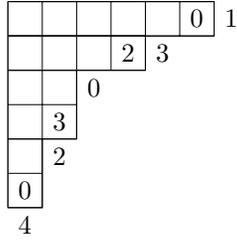
\begin{center}
\Yinternals0\Yremovables1\Yaddables1
\yngres(5,6,4,2^2,1^2)
\end{center}

\noindent Now $A_1$ is a normal $0$-node since there is no addable $0$-node above $A_1$.  Similarly, $A_2$ is a normal $2$-node.  The node $A_3$ is not a normal $3$-node since there is an addable $3$-node $B$ above it but no removable $3$-node $C$ strictly between $A$ and $B$.  Similarly, $A_4$ is not a normal $0$-node.  Both $A_1$ and $A_2$ are good nodes.

  A partition $ \lambda = (\lambda_1, \ldots, \lambda_k)$ is \emph{$p$-regular} if there does not exist $i \geq 1$ such that $\lambda_i = \lambda_{i+1} = \ldots = \lambda_{i+p-1} > 0$.  If such $i$ exists then call $\lambda$ a \emph{$p$-singular} partition.  Finally, $\lambda$ is \emph{$p$-restricted} if $\lambda'$ is $p$-regular.

  \subsection{Hook lengths, $p$-power diagrams, and $p$-JM partitions} \label{subsection:22}  \vskip -.065in
 Suppose $(i,j)$ is a node of the partition $\lambda$.  We define the $(i,j)$-\emph{hook length}, denoted $h_{\lambda}(i,j)$, to be the number of boxes in the Young diagram that sit either to the right of $(i,j)$ in the $i$th row or beneath $(i,j)$ in the $j$th column including $(i,j)$.  More precisely, 
 $h_{\lambda}(i,j) := \lambda_i - i + \lambda_j' - j +1$.  The \emph{hook-length diagram} of $\lambda$ is the Young tableau of shape $\lambda$ where  $h_{\lambda}(i,j)$ is the $(i,j)$-entry.  For any integer $h$, let $\nu_p(h)$ be the largest power of $p$ dividing $h$.  We define the $p$-\emph{power diagram} of $\lambda$ to be the Young tableau of shape $\lambda$ where $\nu_p(h_{\lambda}(i,j))$ is the $(i,j)$-entry.  
 
 \vskip .1in
 \noindent \textbf{Example.}       
The hook-length diagram and 5-power diagram of $\lambda = (6,4,2^2,1^2)$ are given below. (We omit the zeros in the $p$-power diagram.)    

\begin{center}
\young(<11><10><5><4><2><1>,8521,52,41,2,1) \hspace{.5in}
\gyoung(;;;1;;;,;;1;;,1;,;;,;,;) 
\end{center}

The $p$-power diagram of $\lambda$ plays an important role when trying to determine whether an ordinary representation of a symmetric group remains irreducible over a field characteristic $p$.


  \subsection{Some results in the representation theory of the symmetric group}
  If $\lambda$ is a partition of $n$, the associated Specht is denoted by $S^{\lambda}$ (cf. \cite{James}).  The permutation module $M^{\lambda} \cong \text{Ind}_{\Sigma_{\lambda}}^{\Sigma_{n}}(1)$ has the unique indecomposable summand 
 $Y^{\lambda}$ containing $S^{\lambda}$.  The module $Y^{\lambda}$ is called a \emph{Young} module.  Over characteristic zero $\{S^{\lambda} : \lambda \vdash n\}$ is a complete list of non-isomorphic simple modules of $F\Sigma_n$.  Over positive characteristic $p$ the list
$\{D^{\lambda} := S^{\lambda}/ \text{rad} S^{\lambda}: \lambda \vdash n \text{ is $p$-regular}\}$ is a complete list of (non-isomorphic)  simple $F\Sigma_n$-modules (cf. \cite{James}).  

For any $F\Sigma_n$-module $M$, let $M^*$ denote the $F\Sigma_n$-module dual to $M$.  If $M$ has an $F\Sigma_{n}$-module filtration $0 = M_0 \subset M_1 \subset \ldots \subset M_k = M$ with $N_i \cong M_i /M_{i-1}$ we write  
$$M \sim \begin{matrix}  N_k \\ \vdots \\ N_2 \\ N_1 \end{matrix}$$
to represent the filtration.  If the module $M$ is filtered by some modules $N_1, \ldots, N_k$ we may write $M \sim N_1^{a_1} + \ldots + N_k^{a_k}$ where $a_j$ is the multiplicity of the factor $N_j$.  For a simple $F\Sigma_{n}$-module $S$, let $[M:S]$ denote the multiplicity of $S$ as a composition factor of $M$.  Finally, let $sgn: \Sigma_n \rightarrow F$ denote the one-dimensional representation of $\Sigma_n$ defined by: For $\sigma \in \Sigma_n$,
$sgn(\sigma) := 1$ if $\sigma$ is an even permutation or $sgn(\sigma) := -1$ if $\sigma$ is an odd permutation.
We list two well known facts.             

\begin{prop} \label{prop:james} Let $F$ be any field.  
	\begin{enumerate}
		\item (\cite{James}, Theorem 4.9) If $S$ is a simple $F\Sigma_n$-module then $S^* \cong S$.   
		\item (\cite{James}, Theorem 8.15) If $\lambda$ is a partition of $n$ then $S^{\lambda} \otimes sgn \cong (S^{\lambda ' })^*$.   
	\end{enumerate}
\end{prop}

\noindent The next important result is due to Gordon James.

\begin{prop} \label{prop:rad}
Assume $\lambda$ and $\mu$ are $p$-regular partitions of $n$.  Then:
	\begin{enumerate}
		\item $[S^{\lambda}: D^{\lambda}] = 1$, and
		\item $[S^{\lambda} : D^{\mu}] \neq 0$ implies $\mu \unrhd \lambda$.
	\end{enumerate}
\end{prop}

We say that the Specht module $S^{\lambda}$ is a \emph{hook representation} if $\lambda = (i , 1^{n-i})$.  Over odd characteristic hook representations are either irreducible or have two composition factors.  In the latter case the representation is no longer irreducible, yet it remains indecomposable.  We summarize the work of Peel in the next proposition. 
\begin{prop}(\cite{Peel}) \label{prop:peel}
Suppose $p$ is an odd prime.  If $p$ divides $n$ then part of the decomposition matrix of $\Sigma_n$ is given below.
\begin{table}[h!]
\begin{tabular}{c|cccccc}
\hline
$(n)$       &   1 & 0 & 0  & 0 &$\cdots $ & 0\\
$(n-1,1)$  &  1  & 1 & 0  &0 &$\cdots$ & 0\\
$(n-2, 1^2)$ & 0 & 1& 1 &0 & $\cdots $ &0\\
\vdots && &$\ddots$&$\ddots$&&\\ 
&&&&&&\\
$(2, 1^{n-2})$ & 0 &0 & $\cdots$ &&1 & 1\\
$(1^{n})$ & 0 & 0 &$\cdots $&& 0 & 1 
\end{tabular}
\end{table}
\end{prop}
\noindent Peel also proved that the hook representations remain irreducible if $p$ does not divide $n$; however, we only need the case of when $p$ divides $n$ for this paper.  In particular, over odd characteristic hook representations have Loewy length at most 2 and at most 2 composition factors.   

We cite a theorem of Kleshchev and Sheth.

\begin{prop}[\cite{KS}, Theorem 2.10] \label{prop:KS} Let $p > 2$ and let $\lambda$ and $\mu$ be partitions of $n$ with at most $p-1$ nonzero parts.
If $\lambda$ does not strictly dominate $\mu$ then
$$\text{Ext}^1_{\Sigma_n} (D^{\lambda}, D^{\mu}) \cong \text{Hom}_{\Sigma_n} (\text{rad} S^{\lambda}, D^{\lambda}).$$  
\end{prop}

It remains open whether Proposition \ref{prop:KS} holds without the assumption of at most $p-1$ parts.  In Section \ref{section:loewy} we see that one may remove the restriction on the number of parts for $p$-regular partitions in $\mathcal{B}$.

In \cite{M}, James and Mathas introduced a family of partitions, now called $p$-JM partitions, and conjectured that a Specht module $S^{\lambda}$ remains irreducible over a field of odd characteristic $p$ if and only if $\lambda$ is a member of this family of partitions.  This conjecture was proved in \cite{F2}.  We now define what it means to be a $p$-JM partition.  A partition $\lambda$ is a \emph{$p$-JM-partition} if the following property holds: whenever $(i,j)$ is a node of the Young diagram of $\lambda$ and $\nu_p(h_{\lambda} (i,j)) >0$ either all entries in row $i$ of the $p$-power diagram are equal or all entries in column $j$ of the $p$-power diagram are equal.  

\begin{prop}(\cite{F2}) \label{prop:irredSpecht}
Let $p$ be an odd prime.  Then the Specht module $S^{\lambda}$ is irreducible over a field of characteristic $p$ if and only if $\lambda$ is a $p$-JM partition.  
\end{prop}

\vskip .1in
\noindent \textbf{Example.}  Consider the partition $\lambda = (6,4,2^2,1^2)$ from Section \ref{subsection:22}.  We give the 5-power and 3-power diagrams of $\lambda$ below. 

\begin{figure}[h!]
\centering
\subfloat[5-power]{$\gyoung(;;;1;;;,;;1;;,1;,;;,;,;)$
}
\qquad 
\qquad
\subfloat[3-power]{
$\gyoung(;;;;;;,;;;;,;;,;;,;,;)$
}
\end{figure}

\noindent We see that $\lambda$ is not a 5-JM partition, but it is (vacuously) a 3-JM partition.  Hence $S^{(6,4,2^2,1^2)}$ is reducible over characteristic five and irreducible over characteristic three.

\subsection{$p$-cores and $p$-blocks}
The \emph{rim} of $\lambda$ is the collection of all nodes $(i,j)$ of $\lambda$ such that $(i+1,j+1)$ is not a node of $\lambda$.  A \emph{removable rim $p$-hook} for $\lambda$ is a connected sequence of $p$-many nodes in the rim of $\lambda$, say $h \subseteq \lambda$, such that $\lambda - h$ is a partition.  We may repeatedly remove rim $p$-hooks from $\lambda$ to obtain a partition $\overline{\lambda}$ that has no removable rim $p$-hook.  The partition $\overline{\lambda}$ is the \emph{$p$-core} of $\lambda$.  The number of $p$-hooks removed, say $w$, until we reach the $p$-core is called the \emph{$p$-weight} (or weight) of the partition.

\vskip .1in
\noindent \textbf{Example.}  Let $p=5$ and $\lambda = (6,4,2)$.  The Young diagram for $\lambda$ is given below, and the collection of boxes marked $\star$ represents the rim of $\lambda$.

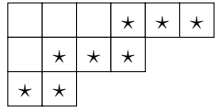
\begin{center}
\gyoung(;;;;\star\star\star,;;\star\star\star,\star\star)
\end{center}

\noindent The following sequence demonstrates the repeated removal of rim 5-hooks until the 5-core is reached.  The collection of boxes marked $\bullet$ represents the rim 5-hook that is to be removed at each stage.  
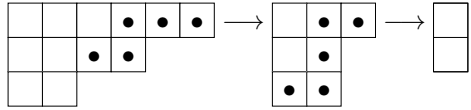
\begin{center}
$
\gyoung(;;;;\bullet\bullet\bullet,;;;\bullet\bullet,;;) 
\longrightarrow
\gyoung(;;\bullet\bullet,;;\bullet,\bullet\bullet)
\longrightarrow
\yng(1^2)
$
\end{center}

\noindent So $\lambda = (6,4,2)$ has 5-core $\overline{\lambda} = (1^2)$.  Obtaining the 5-core of $\lambda$ required removing a total of 2 rim 5-hooks, so the weight of $\lambda$ is 2.

The $p$-blocks of the symmetric group are classified by a theorem commonly referred to as `The Nakayama Conjecture.'  The conjecture of Nakayama was first proved simultaneously by Brauer and Robinson.  We state the result in the next proposition.   

\begin{prop}[The Nakayama Conjecture] \label{prop:N}
The Specht modules $S^{\lambda}$ and $S^{\mu}$ lie in the same $p$-block of $F\Sigma_{n}$ if and only if $\lambda$ and $\mu$ have the same $p$-core.  
\end{prop}  

\noindent Thus the $p$-blocks of the symmetric group are labeled by $p$-cores.  The $p$-core of $\lambda$ corresponds to a $p$-block $B$ which contains the Specht module $S^{\lambda}$ (and all of its composition factors).  In this case, we say $\lambda$ is in the block $B$.  Now suppose $\mu$ is a partition of $n$ with the same $p$-core as $\lambda$ and also assume $\lambda$ has weight $w$.  Then $\mu$ is in $B$, and $\mu$ is also a weight $w$ partition of $n$.  Hence all partitions of $n$ in $B$ necessarily have the same $p$-weight $w$.  We say $w$ is the weight of the block $B$.  

The principal block, $\mathcal{B}$, of $F\Sigma_{3p}$ is the unique $p$-block of $F\Sigma_{3p}$ containing the trivial module $F \cong S^{(3p)}$.  By the Nakayama conjecture, a Specht module $S^{\lambda}$ is in $\mathcal{B}$ if and only if $\lambda$ is a partition of $3p$ that has empty $p$-core.  Equivalently, $\lambda$ is a weight 3 partition of $n = 3p$.  We give a fact about the decomposition numbers for the symmetric group for weights 2 and 3 in the next proposition.     

\begin{prop} (\cite{Scopes} Theorem I,  \cite{F3} Theorem 1.1)\label{prop:decnumb}
Suppose $\lambda$ and $\mu$ are weight $w$ partitions of $n$ with $\mu$ $p$-regular and $w = 2$ or $3$.  Then
$[S^{\lambda} : D^{\mu}] \leq 1$.   
\end{prop}           

 \subsection{Abacus display and pyramids}
 A convenient way to display a partition is on an abacus with vertical runners labeled $1, \ldots, p$ (from left to right).  We will refer to this labeling of the runners as the \emph{left-to-right} numbering.  We adopt the convention of calling the first position (upper left-hand corner) on the the display `position 1'.  So the positions on the first column of the abacus are $1, 1+p, 1+2p$, et cetera, and the positions of the first two rows are $1, 2, \ldots, p, 1+p, 2+p, \ldots, 2p$.  For example, if $p =5$ then the positions of the abacus are as follows.
 
 \begin{center}
 $
 \begin{array}{ccccc}
 \hline
 1& 2& 3& 4& 5\\
 6 & 7 & 8 & 9 & 10\\
 11& 12 & 13 & 14 & 15\\
 16& 17 & 18 & 19 & 20\\
 \vdots & \vdots & \vdots & \vdots& \vdots
 \end{array}  
$
\end{center} 
\noindent Now suppose $a$ and $b$ are positions on an abacus display for $\lambda$.  If $a < b$, we either say that position $b$ comes after position $a$ or position $a$ comes before position $b$.    

 Suppose the partition $\lambda = (\lambda_1, \ldots, \lambda_k)$ has $k$ non-zero parts.  For an integer $r \geq k$ define the \emph{beta-number} $\beta_i = \lambda_i +r - i +1$ for $i = 1, \ldots, r$ with the convention that if $r > k$ then $\lambda_j = 0$ for $k< j \leq r$.
We may use this set of beta-numbers for $\lambda$ to form an abacus display for $\lambda$ which has a bead `$\bullet$' at position $\beta_i$ for each $i = 1,\ldots, r$ and `$\circ$' at any position $k \geq 1$ that is not a member of the specified set of beta-numbers.  We call a position of the display having `$\bullet$' \emph{occupied} and a position with 
`$\circ$' \emph{unoccupied} (or empty).  Note that one may see the beta-number defined as: $\beta_i  = \lambda_i + r -i$.  Our definition ensures the beta-numbers for partitions in $\mathcal{B}$ are at least one when $r = 3p$. 

\vskip .1in
\noindent \textbf{Example.} Let $p =5$, $\lambda = (7^2,2^2,1)$, and $r  =15$.  We represent $\lambda$ on an abacus with 15 beads.  Then a set of beta-numbers of $\lambda$ is: $\beta_1 = 22$, $\beta_2 = 21$, $\beta_3 = 15$, $\beta_4 = 14$, $\beta_5 = 12$,and 
$\beta_j = 0 +15 - j +1 = 16 - j$ for $6 \leq j \leq 15$.  For each $1 \leq i \leq r = 15$ we place a bead at position $\beta_i$.  The abacus display for $\lambda$ is given in Figure \ref{fig:fig1}.      
 
 \begin{figure}[ht!] 
 \begin{center}
 $
 \begin{array}{ccccc}
 1 & 2 &3 & 4 & 5\\
 \hline
 \bullet & \bullet & \bullet & \bullet & \bullet\\
 \bullet & \bullet & \bullet & \bullet & \bullet\\
 \circ & \bullet & \circ & \bullet & \bullet\\
 \circ & \circ & \circ & \circ & \circ\\
 \bullet & \bullet & \circ &\circ & \circ \\
 \circ& \circ & \circ &\circ & \circ\\
 \vdots & \vdots & \vdots & \vdots& \vdots
 \end{array}  
$
\caption{Abacus display for $\lambda$}
\label{fig:fig1} 
\end{center}     
\end{figure}
Fix an abacus display for the partition $\lambda$.  Suppose the display has a bead at position $m > 1$ and no bead at position $m-1$.  Pushing said bead `to the left' to position $m-1$ corresponds to removing a (removable) node from $\lambda$.  Call such a bead \emph{removable}.  Now suppose there is a bead at position $m$ and no bead at position $m+1$.  Pushing the bead `to the right' to position $m+1$ corresponds to adding an addable node to $\lambda$.  Call such a bead \emph{addable}.  Finally, we call a bead on the display at position $x$ \emph{proper} if there is an unoccupied position before position $x$.  Otherwise, we call a bead \emph{improper}.   For example, referring to Figure \ref{fig:fig1}, we see that the abacus display has removable beads at positions 12, 14, and 21 and addable beads at positions 12, 15, and 22.  Also, the beads at positions 1 through 10 are all improper while all other beads are proper.  Suppose an abacus display for $\lambda$ has a removable bead in row $r$ of runner $i$.  If $i \geq 2$ we call the removable bead \emph{normal} provided that for any $j \geq 1$ the number of beads on runner $i$ in rows $r+1, \ldots, r+j$ is at least the number of beads on runner $i-1$ in rows $r+1, \ldots, r+j$.  If $i=1$ we call the bead \emph{normal} provided that for any $j \geq 1$ the number of beads on runner $i=1$ in rows $r +1, \ldots, r+j$ is at least the number of beads on runner $p$ in rows $r, \ldots, r+j-1$.  We call such a bead \emph{normal} because it corresponds to a normal node of $\lambda$.              

Now suppose the display has a bead at position $m > p$ and no bead at position $m  - p$ (i.e., one row above the bead on the same runner).  Pushing the bead `up' to position $m -p$ corresponds to removing a (removable) rim $p$-hook from $\lambda$.  Thus we obtain an abacus display for the $p$-core of $\lambda$ by pushing the beads up as far as possible.  Similarly, if there is a bead at position $m$ and no bead at position $m+p$ then pushing the bead at position $m$ `down' to position $m+p$ corresponds to adding an addable rim $p$-hook to $\lambda$.  The beads in Figure \ref{fig:fig1} corresponding to removable rim 5-hooks of $\lambda = (7^2,2^2,1)$ are at positions 21 and 22.       

Given an abacus display for the partition $\lambda$ we may obtain an abacus display for the $p$-core of $\lambda$ (with the same number of beads) by `pushing up' the beads on the display as far as possible.  If the display has a multiple of $p$ beads then one records the mentioned process on a $p$-quotient which is a sequence of partitions $[\lambda^{(1)}, \ldots, \lambda^{(p)}]$ where $\lambda^{(j)}$ is defined as follows.  If you cannot push a bead up on runner $j$ then $\lambda^{(j)} := \emptyset$.  Suppose $b_{j_1}, \ldots, b_{j_k}$ are the beads on runner $j$ that can be pushed up.  We define $\lambda^{(j)}_m$ to be the number of unoccupied positions on runner $j$ above $b_{j_m}$.  We call $\lambda^{(j)}_{m}$ the \emph{weight} of the bead $b_{j_m}$.  We refer the reader to Section 2.7 of \cite{JK} for a detailed discussion on the $p$-quotient, but we note here that if $\lambda$ has $p$-weight $w$ then $\sum_{i=1}^{p} |\lambda^{(i)}| = w$.  For example, using Figure \ref{fig:fig1}, we see that $\lambda= (7^2, 2^2,1)$ has 5-weight 3 and 5-quotient $[(2), (1), \emptyset, \emptyset, \emptyset]$.

The $p$-quotient will be useful in determining which Specht modules in $\mathcal{B}$ are irreducible (see Theorem \ref{thm:jm}).  Fayers provides a test involving a $p$-quotient of the partition $\lambda$ (with the left-to-right numbering) in Proposition 2.1 of \cite{F2} which determines whether $\lambda$ is a $p$-JM partition.  It will be helpful to use a different form of the test given by Fayers in \cite{F}.  In order to state the test as in \cite{F} we need to define a special numbering on the runners of the abacus and also define the `pyramid' of a block.

Fix an abacus display for $\lambda$ that has a multiple of $p$ beads with runners $1, \ldots, p$.  Push all beads up as far as possible to obtain an abacus display, $\Gamma$, for the $p$-core of $\lambda$.  For each runner $i$, let $q_i$ be the  position of the first empty space in $\Gamma$ on runner $i$.  Now write the $q$'s in ascending order $q_1 < q_2 < \ldots < q_p$ (reindexing if needed) and renumber the runners of the display so that position $q_i$ is on runner $i$.  We refer to this renumbering of the runners as the \emph{non left-to-right numbering}.  Observe that we have renumbered the runners from left-to-right as $\sigma(1), \sigma(2), \ldots, \sigma(p)$ for the unique permutation $\sigma \in \Sigma_{p}$.  For $1 \leq k < \ell \leq p$, define
$$B^{k}_{\ell} = \left\lfloor \frac{q_{\ell} - q_{k}}{p} \right\rfloor.$$ 
The array $(B^{k}_{\ell})$ is called the \emph{pyramid} of the block in which $\lambda$ lies.  Define the $p$-quotient of $\lambda$ in the non left-to-right numbering to be the sequence $[\mu^{(1)}, \mu^{(2)}, \ldots, \mu^{(p)}]$ where 
$\mu^{(j)}:= \lambda^{(\sigma^{-1}(j))}$.  We refer to this $p$-quotient of $\lambda$ as the \emph{reordered} $p$-quotient of $\lambda$.        

\vskip .1in
\noindent \textbf{Example.} Let $\lambda = (7^2,2^2,1)$ and $p=5$.  We use the abacus display for $\lambda$ from Figure \ref{fig:fig1}.  Pushing all beads up on the display we obtain an abacus display for the $5$-core (see Figure \ref{fig:reorder}(A)), and see $\lambda$ has $5$-quotient $[(2), (1), \emptyset, \emptyset, \emptyset]$ in the left-to-right numbering.  Now 
$(q_1, q_2, q_3, q_4, q_5) = (13,16,19,20,22)$.  We have $B^{1}_{3} = B^{1}_{4} = B^{1}_{5} =  B^{2}_{5}= 1$ and all other $B^{k}_{\ell}$ are zero.  The runners are renumbered from left to right as 2, 5, 1, 3, and 4.  The renumbering is given by the permutation $\sigma= (12543) \in \Sigma_{5}$.  So the reordered 5-quotient of $\lambda$ is
$$[\mu^{(1)}, \mu^{(2)}, \mu^{(3)}, \mu^{(4)}, \mu^{(5)}] = 
[ \lambda^{(3)}, \lambda^{(1)}, \lambda^{(4)}, \lambda^{(5)}, \lambda^{(2)}] = 
[\emptyset, (2), \emptyset, \emptyset, (1)].$$

\begin{figure}[h!]
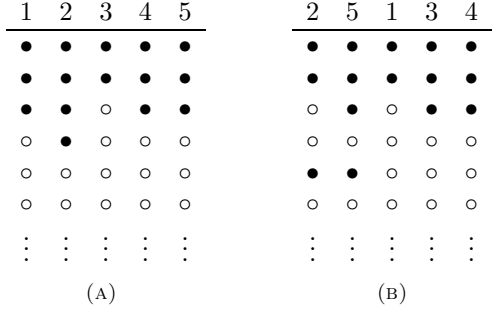

\centering
\subfloat[]{
 $
 \begin{array}{ccccc}
 1 & 2 &3 & 4 & 5\\
 \hline
 \bullet & \bullet & \bullet & \bullet & \bullet\\
 \bullet & \bullet & \bullet & \bullet & \bullet\\
 \bullet & \bullet & \circ & \bullet & \bullet\\
 \circ & \bullet & \circ & \circ & \circ\\
 \circ & \circ & \circ &\circ & \circ \\
 \circ& \circ & \circ &\circ & \circ\\
 \vdots & \vdots & \vdots & \vdots& \vdots
 \end{array}  
$}
\quad
\quad
\quad
\subfloat[]{
 $
 \begin{array}{ccccc}
 2 & 5 &1 & 3 & 4\\
 \hline
 \bullet & \bullet & \bullet & \bullet & \bullet\\
 \bullet & \bullet & \bullet & \bullet & \bullet\\
 \circ & \bullet & \circ & \bullet & \bullet\\
 \circ & \circ & \circ & \circ & \circ\\
 \bullet & \bullet & \circ &\circ & \circ \\
 \circ & \circ & \circ &\circ & \circ\\
 \vdots & \vdots & \vdots & \vdots& \vdots
 \end{array}  
$
}
\caption{Reordering for $\lambda = (7^2,2^2,1)$}
\label{fig:reorder}
\end{figure} 
We are now ready to state the test due to Fayers.  
\begin{prop} [\cite{F}, Proposition 4.9] \label{prop:FJM}
Let $p$ be an odd prime.  Take an abacus display for the partition $\lambda$ with non left-to-right numbering, and suppose the corresponding reordered $p$-quotient is $[\mu^{(1)}, \ldots, \mu^{(p)}]$.  Then $\lambda$ is a $p$-JM partition if and only if:
	\begin{enumerate}
		\item $\mu^{(i)} = \emptyset$ for $i \neq 1, p$.
		\item $\mu^{(1)}$ is a $p$-restricted $p$-JM partition.
		\item $\mu^{(p)}$ is a $p$-regular $p$-JM partition.
		\item $(\mu^{(k)})_1 + ((\mu^{(\ell)})')_1 \leq B^k_{\ell} + 1$ for all $k < \ell$.
	\end{enumerate}
\end{prop}
\noindent Next we define abacus notation for partitions in $\mathcal{B}$.

By the Nakayama conjecture, $S^{\lambda}$ is in the principal block of $F\Sigma_{3p}$ if and only if $\lambda$ has an empty $p$-core.  The empty $p$-core can be represented by an abacus display with three beads on each runner (so the display has $r = 3p$ beads).  Hence $\lambda$ has an abacus display with three beads on each runner (the $\langle 3^p \rangle$ display for $\lambda$).  We use the $\langle 3^p \rangle$ abacus display to study partitions in $\mathcal{B}$.

\begin{defn} \label{def:3p}
Let $\lambda$ be a partition in the principal block of $F\Sigma_{3p}$.
	\begin{enumerate}
		\item Write $\langle  i \rangle$ if the abacus display for $\lambda$ has one bead of weight 3 on runner $i$.  
		\item Write $\langle i,j \rangle$ if the abacus display for $\lambda$ has one bead of weight 2 on runner $i$ and one bead of weight one on (not necessarily distinct) runner $j$.
		\item Write $\langle i,j,k \rangle$ if the abacus display for $\lambda$ has one bead of weight 1 on (not necessarily distinct) runner $i$,$j$, and $k$.    
	\end{enumerate}   
\end{defn}

\vskip .1in
\noindent \textbf{Example.} The $\langle 3^{7} \rangle$ display for the partition $\lambda = (13,4,2,1^2)$ is given below.  So $\lambda$  is represented in  $\langle 3^{7} \rangle$ notation by $\langle 6,3 \rangle$.  

\begin{center}
$
\begin{array}{ccccccc}
   1        &    2        &    3        &     4         &       5          &    6              &      7 \\
\hline  
\bullet   & \bullet   &  \bullet   &   \bullet   &    \bullet     &   \bullet       & \bullet \\
\bullet   & \bullet   &  \bullet   &   \bullet   &    \bullet     &   \bullet       & \bullet \\
\bullet   & \bullet   &    \circ    &   \bullet   &    \bullet     &   \circ          & \bullet \\
\circ     &   \circ     &  \bullet   &    \circ     &      \circ     &   \circ           & \circ    \\
\circ     &   \circ     &  \circ      &    \circ     &      \circ     &   \bullet           & \circ    \\
\circ     &   \circ     &  \circ      &    \circ     &      \circ     &   \circ          & \circ    \\
\vdots  &  \vdots   & \vdots    & \vdots     &  \vdots      &  \vdots        &\vdots   \\
\end{array}
$
\end{center}

For a partition $\lambda$ let $\tau (\lambda)$ denote the number of removable nodes for $\lambda$ and let $\tau_{p}(\lambda)$ denote the number of normal nodes for $\lambda$.  One may easily determine $\tau(\lambda)$ and $\tau_p(\lambda)$ from an abacus display for $\lambda$ with $p$-many runners.

\vskip .1in
\noindent \textbf{Example.} Take $p = 5$, and consider $\lambda = \langle 3,5 \rangle $ and $\mu = \langle 5,3,1 \rangle$ in $\mathcal{B}$.  

\begin{figure}[h!]

\subfloat[$\langle 3,5 \rangle$]{
$
\begin{array}{ccccc}
1            &        2           &           3            &           4          &          5 \\
\hline
\bullet    &   \bullet        &      \bullet          &    \bullet          &    \bullet \\ 
\bullet    &   \bullet        &      \bullet          &    \bullet          &    \bullet \\ 
\bullet    &   \bullet        &      \circ             &\boxed{\bullet} &       \circ \\ 
\circ       &     \circ         &      \circ             &     \circ            &\boxed{\bullet} \\
\circ       &     \circ         &\boxed{\bullet}   &      \circ           &    \circ    
\end{array}
$
}
\qquad
\qquad
\subfloat[$\langle 5,3,1 \rangle$]{
$
\begin{array}{ccccc}
1                        &        2               &           3            &           4            &          5 \\
\hline
\bullet                &   \bullet             &      \bullet          &    \bullet            &    \bullet \\ 
\bullet                &   \bullet             &      \bullet          &    \bullet            &    \bullet \\ 
\circ                   & \boxed{\bullet}  &      \circ             &\boxed{\bullet}   &       \circ \\ 
\boxed{\bullet}   &     \circ              &\boxed{\bullet}   &     \circ              &\boxed{\bullet} \\
\circ                   &     \circ             &     \circ               &      \circ              &    \circ    
\end{array}
$
}
\caption{$\langle 3^5 \rangle$ displays for $\lambda$ and $\mu$}
\label{fig:remove}
\end{figure}  

\noindent The removable beads of each display are marked with a box.  So $\tau(\lambda) = 3$ and $\tau(\mu) = 5$.  In Figure \ref{fig:remove}(A) we see that removable beads at positions 20 and 23 correspond to normal nodes while the bead at position 14 does not.  Hence $\tau_5(\lambda) = 2$.  Next, in Figure \ref{fig:remove}(B), the beads at positions 18 and 20 corresponds to normal nodes of $\mu$.  So $\tau_5(\mu) =2$.

Analysis of the $\langle 3^{p}\rangle$ displays gives the next proposition.  
\begin{prop} \label{prop:nremovables}
We have the following. 
\vskip .1in
\noindent (1) For $1 \leq i \leq p-1$
\vskip .1in
$\tau(\langle i,p \rangle) = \begin{cases} 2 \text{ if $i = 1, p-1$}\\ 3 \text{ if $2 \leq i \leq p-2$}  \end{cases}$ and 
		$\tau_{p}(\langle i,p \rangle) = \begin{cases} 1 \text{ if $i = 1, p-1$}\\ 2 \text{ if $2 \leq i \leq p-2$}  \end{cases}$
\vskip .1in
\noindent (2) For $2 \leq i \leq p$
\vskip .1in
$\tau (\langle i,1 \rangle) = \begin{cases} 3 \text{ if $3 \leq i \leq p$}  \\ 2 \text{ if $i = 2$ } \end{cases}$ and \ 
		$\tau_p(\langle i,1 \rangle) =  1 \text{ for all $2 \leq i \leq p$} $
\vskip .1in
\noindent  (3) For $2 \leq j \leq p-1$ 
\vskip .1in
$\tau(\langle j+1 , j \rangle) = \begin{cases} 2 \text{ if $j = p-1$}\\ 3  \text{ if $j \leq p-2$}\end{cases}$ and 
		$\tau_p(\langle j+1 , j \rangle) =  2  \text{ for all  $j$} $ 
\vskip .1in
\noindent (4) For $2 \leq j < i \leq p$ and $i-j \geq 2$
\vskip .1in
$\tau(\langle i , j \rangle ) = \begin{cases} 3 \text{ if $i = p$}\\ 4  \text{ if $ i  \leq p-1$}\end{cases}$ 
		and $\tau_p(\langle i,j \rangle) =  2 \text{ for all  $3 \leq i  \leq p$}$
\vskip .1in
\noindent (5) For any $p$-regular partition $\lambda$ in $\mathcal{B}$:
			\begin{enumerate}
			  \item [(a)] If $p = 5$ then $\tau_5(\lambda) \leq 2$.
                  \item [(b)] For $p \geq 7$, $\tau_p(\lambda) = 3$ if $\lambda = \langle i,j,k\rangle$ with $i > j > k \geq 2$ and $i - j \geq 2$, $j-k \geq 2$ or $\tau_{p}      (\lambda) \leq 2$ otherwise.
                 \end{enumerate}
\end{prop}


\subsection{The Mullineux map and parity} 
Let $\lambda$ be a $p$-regular partition of $n$.  Mullineux conjectured there is a simple algorithm to determine the $p$-regular partition $m(\lambda)$ of $n$ such that $D^{\lambda} \otimes sgn \cong D^{m(\lambda)}$.  Ford and Kleshchev proved the conjecture in \cite{FK}.  The map $m$ is a bijection from the set of $p$-regular partitions of $n$ onto itself, and it is called the \emph{Mullineux map}.  The next proposition contains the defining property of the Mullineux map.

\begin{prop}[\cite{FK}, Theorem 8.7] \label{prop:mullk}
Suppose $A$ is a good node for a $p$-regular partition $\lambda$ of $n$, with $\text{res}A = \alpha$.  Then there exists a good node $B$ for $m(\lambda)$ with $\text{res}B = -\alpha$ and $m(\lambda_A) = m(\lambda)_B$.  
\end{prop}

\noindent By definition, a good node for $\lambda$ is normal, but a normal node for $\lambda$ need not be good.  If the normal nodes for $\lambda$ are of distinct $p$-residues then the set of good nodes and the set normal nodes for $\lambda$ coincide.  If $\lambda$ is a $p$-regular partition in the principal block of $F\Sigma_{3p}$ then $\{ \text{good nodes for } \lambda \}  = \{ \text{normal nodes for } \lambda \}$ since each removable node is of a distinct $p$-residue.  That the removable nodes have distinct residues is easily deduced after analyzing the $\langle 3^p \rangle$ displays of Definition \ref{def:3p}.           

A few proofs require computing the image of a partition under the Mullineux map.  If $\Lambda$ is the $\langle 3^p \rangle$ abacus display for the 
$p$-regular partition $\lambda$ in $\mathcal{B}$, let $m(\Lambda)$ denote the $\langle 3^p \rangle$ abacus display for $m(\lambda)$.  We follow the procedure given by Ford and Kleshchev to compute the effect the Mullineux map has on an abacus display (see \cite{FK} Definition 1.3 up to Definition 1.6).

We now define the parity of a partition $\lambda$.  Suppose we remove rim $p$-hooks $h_1, \ldots, h_{w}$ to obtain the $p$-core for $\lambda$.  Define the \emph{leg-length} of $h_i$ to be $\ell(h_i) :=  (\text{the number of rows in $h_i$} ) - 1$.  Now define $\ell : = \sum_{i=1}^{w}\ell(h_i)$.  The \emph{parity} of $\lambda$, denoted $\mathcal{P} \lambda$,  is $(-1)^{\ell}$.  We cite two important results of Fayers and Tan.  

\begin{prop}[\cite{FT}, Proposition 2.19] \label{prop:mullparity}
For any $p$-regular partition $\lambda$ of weight 3, we have $\mathcal{P}(m(\lambda)) \neq \mathcal{P}\lambda$.     
\end{prop}  

\begin{prop}(\cite{FT}, Theorem 1.1) \label{prop:bipart}
Suppose that $F$ is a field of characteristic at least 5 and that $B$ is a weight 3 block of $F\Sigma_n$.  If $\lambda$ and $\mu$ are $p$-regular partitions in $B$ with $\mathcal{P}\lambda = \mathcal{P} \mu$, then $\text{Ext}^{1}_{B}(D^{\lambda}, D^{\mu}) = 0$.  In particular, the $\text{Ext}^1$-quiver of $B$ is bipartite.     
\end{prop} 

Suppose $\lambda$ is a partition of $n$ that is not a $p$-core, and let $h$ be a removable rim $p$-hook for $\lambda$.  It is not hard to see that if $h$ has $r$ rows then $h$ has $p-r+1$ columns.  So $h$ has leg-length $r-1$, and $h$ corresponds to a removable rim $p$-hook, say $h'$, for $\lambda'$ with leg-length $p-r+1 - 1 = p-r$.  From this we deduce $\lambda$ and $\lambda'$ have the same parity when $p$ is odd.  This observation will be applied later in the proof of Theorem \ref{thm:RR}.           

It is well known that the Young module $Y^{\lambda}$ is self-dual and indecomposable.  Moreover, if $\lambda$ is $p$-restricted then $Y^{\lambda}$ has the simple socle isomorphic to $D^{m(\lambda ')}$ and $Y^{\lambda}$ is projective.  In \cite{HN}, Hemmer and Nakano show that the multiplicities in a Specht or dual Specht module filtration of an $F\Sigma_{n}$-module are well defined.  The authors also show that the $F\Sigma_{n}$- module $Y^{\lambda}$ has a Specht module filtration of the form:
$$0 \subset S^{\lambda} = M_1 \subset M_2 \subset \ldots \subset M_{k-1} \subset M_{k} = Y^{\lambda}$$
where $M_{i}/M_{i-1} \cong S^{\mu_{i}}$ and $\mu_{i} \rhd \lambda$ for all $2 \leq i \leq k$.  We give two important results about Young modules in the next proposition.      

\begin{prop} \label{prop:Young}
Let $\lambda$ be a partition of $n$.  
\begin{enumerate}
	\item If $\lambda$ is $p$-restricted then the Young module $Y^{\lambda}$ is isomorphic to the projective cover of $D^{m(\lambda')}$.
	\item If $p \geq 5$ and $\lambda$ is $p$-regular then $[Y^{\lambda}: D^{\lambda}] = 1$.  
\end{enumerate}
\end{prop}

\subsection{Ext$^1$-quiver, Branching, and Carter-Payne maps} In \cite{MR}, Martin and Russell construct the $\text{Ext}^1$-quiver for $\mathcal{B}$ from $\text{Ext}^1$-quivers for defect 2 and defect 1 blocks of $F\Sigma_{3p-1}$.  In doing so, the authors establish rules for restricting to defects 2 and 1 as well as rules for inducing back up to $\mathcal{B}$.  For $2 \leq i \leq p$, let $B_i$ denote the defect 2 block of $F\Sigma_{3p-1}$ corresponding to the $p$-core $(i-1,1^{p-i})$, and let $B_1$ be the defect 1 block of $F\Sigma_{3p-1}$ with $p$-core $(p,1^{p-1})$.  We now turn our attention to how these blocks relate to the removable beads on the $\langle 3^p \rangle$ display.       

Fix the $\langle 3^p \rangle$ display for a partition $\lambda$.  Suppose the display has a removable bead on runner 
$1 \leq i \leq p$ at position $m$.  Pushing said bead from position $m$ to position $m-1$ produces a non-zero restriction of the Specht module $S^{\lambda}$ to the block $B_i$.  Further, we remove a removable $(i-1)$-node from $\lambda$ producing a partition $\tilde{\lambda}$ of $3p-1$.  In this case, write $S^{\lambda}\downarrow_{B_i} \cong S^{\tilde{\lambda}}$.  Martin and Russell showed the following: If the $\langle 3^p \rangle$ display for $\lambda$ has a removable bead on runner $2 \leq i \leq p$ then there is exactly one partition $\sigma_i(\lambda)$ in $\mathcal{B}$ such that 
$S^{\lambda} \downarrow_{B_i} \cong S^{\tilde{\lambda}} \cong S^{\sigma_i(\lambda)} \downarrow_{B_i}$ and $\lambda > \sigma_i(\lambda)$.  It is easily seen that we also have $\lambda \rhd \sigma_i(\lambda)$.  Conversely, if $\tilde{\lambda}$ is a partition in the block $B_i$ of $F\Sigma_{3p-1}$ and $2 \leq i \leq p$ then there are exactly two partitions 
$\lambda$ and $\mu$ in $\mathcal{B}$ such that $S^{\lambda} \downarrow_{B_i} \cong S^{\mu}\downarrow_{B_i} \cong S^{\tilde{\lambda}}$ with $\lambda > \mu$.  It is not hard to see that $\mu = \sigma_i(\lambda)$.  If the $\langle 3^p \rangle$ for $\lambda$ has a removable bead on runner $i=1$ then there are two partitions $\mu$ and $\nu$ in $\mathcal{B}$ such that $\lambda > \mu > \nu$ and 
$S^{\lambda} \downarrow_{B_1} \cong S^{\mu} \downarrow_{B_1} \cong S^{\nu} \downarrow_{B_1} \cong S^{\tilde{\lambda}}$.  Conversely, if the partition $\tilde{\lambda}$ is in $B_1$ then there are three partitions $ \lambda >\mu > \nu$ in $\mathcal{B}$ so that $S^{\lambda} \downarrow_{B_1} \cong S^{\mu} \downarrow_{B_1} \cong S^{\nu} \downarrow_{B_1} \cong S^{\tilde{\lambda}}$   

For each $i =1, \ldots, p$, define a subset of the set of $p$-regular partitions in $\mathcal{B}$, 
$$\Lambda_i = \{\lambda \  | \ D^{\lambda}\downarrow_{B_i} \neq0 \}.$$

\noindent By Kleshchev's modular branching rules for simple modules (see Theorem $\text{E}'$ of \cite{BK2}), a $p$-regular partition $\lambda$ is in $ \Lambda_i$ if and only if the $\langle 3^p \rangle$ display for $\lambda$ has a normal bead on runner $i$.  If $\lambda \in \Lambda_i$ then $D^{\lambda}\downarrow_{B_i} \cong D^{\tilde{\lambda}}$ is simple.  Moreover, it is known that $D^{\tilde{\lambda}} \uparrow^{\mathcal{B}}$ is self-dual with simple head and simple socle both isomorphic to $D^{\lambda}$.  By Theorem 3.9 of \cite{MR}, if $2 \leq i \leq p $ and $\sigma_i(\lambda)$ is $p$-regular then $D^{\sigma_i(\lambda)}\downarrow_{B_i} = 0$ and $[D^{\tilde{\lambda}} \uparrow^{\mathcal{B}}: D^{\sigma_i(\lambda)}] = 1$.  Now, for $1 \leq i \leq p$, define a function $\Theta_i$ from $\Lambda_i$ to the set of partitions of $3p-1$ lying in the block $B_i$ as follows: 
For $\lambda \in \Lambda_i$, define $\Theta_i(\lambda)$ to be the partition $\tilde{\lambda}$ in $B_i$ such that 
$S^{\lambda}\downarrow_{B_i} \cong S^{\tilde{\lambda}}$.  Martin and Russell prove that $\lambda \in \Lambda_i$ is $p$-singular if and only if $\Theta_i(\lambda)$ is $p$-singular.  Furthermore, if $\lambda$ and $\mu$ are in
$\Lambda_i$ and $\lambda > \mu$ then $\Theta_i(\lambda) > \Theta_i(\mu)$.  We gather results of Martin, Russell, and Tan in the following proposition.
\begin{prop} \label{prop:MT} Let $\lambda$ and $\mu$ be $p$-regular partitions in $\mathcal{B}$.
	\begin{enumerate}
		\item Suppose $\lambda \in \Lambda_i$ ($1\leq i \leq p$).  Then $D^{\lambda}$ extends $D^{\mu}$ if and only if either $\mu \in \Lambda_i$ and $D^{\tilde{\lambda}}$ extends $D^{\tilde{\mu}}$ or $\mu \notin \Lambda_i$ and $[D^{\tilde{\lambda}} \uparrow^{\mathcal{B}} : D^{\mu}] \neq 0$. 
		\item Suppose $D^{\tilde{\lambda}}$ is a simple module of $B_i$ for some $i$ ($2 \leq i \leq p$).  
		Then $D^{\tilde{\lambda}} \uparrow^{\mathcal{B}}$ has Loewy length 3.
		\item Suppose $D^{\tilde{\lambda}}$ extends $D^{\tilde{\mu}}$ in $B_i$ for some $i$ ($1 \leq i \leq p$).  Then $D^{\tilde{\lambda}} \uparrow^{\mathcal{B}}$ extends $D^{\mu}$, where $D^{\mu} = \text{soc}(D^{\tilde{\mu}}\uparrow ^{\mathcal{B}})$, and $D^{\mu}$ occurs in the second Loewy layer of the non-split extension of $D^{\tilde{\lambda}} \uparrow^{\mathcal{B}}$ by $D^{\mu}$.
		\item Suppose $D^{\tilde{\lambda}}$ extends $D^{\tilde{\mu}}$ in $B_i$ for some $i$ ($2 \leq i \leq p$).  Let $\tilde{M}$ be a non-split extension of $D^{\tilde{\lambda}}$ by $D^{\tilde{\mu}}$.  Then $\tilde{M}\uparrow^{\mathcal{B}}$ has Loewy length 4.    
	\end{enumerate}
\end{prop}

\begin{cor} \label{cor:CPext}
Suppose $\lambda$ is a $p$-regular partition  in $\Lambda_i$ ($2 \leq i \leq p$).  If $\sigma_i(\lambda)$ is $p$-regular then $\text{Ext}^1_{\Sigma_{3p}}(D^{\lambda} , D^{\sigma_i(\lambda)}) \neq 0$.    
\end{cor}

\begin{proof}
Since $\lambda \in \Lambda_i$, $D^{\lambda} \downarrow_{B_i} \cong D^{\tilde{\lambda}}$ is simple.  By Proposition \ref{prop:MT}(2), 
$D^{\tilde{\lambda}} \uparrow^{\mathcal{B}}$ has Loewy length 3 with simple head and simple socle isomorphic to $D^{\lambda}$.  Now 
$[D^{\tilde{\lambda}} \uparrow^{\mathcal{B}} : D^{\sigma_i(\lambda)}] = 1$ and $\lambda \rhd \sigma_i(\lambda)$, so $D^{\sigma_i(\lambda)}$ must lie in the second Loewy layer of $D^{\tilde{\lambda}} \uparrow^{\mathcal{B}}$.  Hence, 
$\text{Ext}^1_{\Sigma_{3p}}(D^{\lambda}, D^{\sigma_i(\lambda)}) \neq 0$.     
\end{proof}

Part (1) of the proposition is Proposition 3.1 of \cite{MT} which is a summary of the work of Martin and Russell in \cite{MR}.  Parts (2), (3), and (4) are Proposition 3.3, Lemma 3.4, and Proposition 3.5 of \cite{MT}, respectively.  An important consequence of Proposition \ref{prop:MT} (2),(3),(4) is the next lemma.

\begin{lem} \label{lem:sinduce}
Suppose the Specht module $S^{\widetilde{\lambda}}$ lies in the block $B_i$ of $F\Sigma_{3p-1}$ where $2 \leq i \leq p$.  We have the following.  
	\begin{enumerate}
		\item If $S^{\widetilde{\lambda}}$ has Loewy length 1 then $S^{\widetilde{\lambda}} \uparrow^{\mathcal{B}}$ has Loewy length 3.  
		\item If $S^{\widetilde{\lambda}}$ has Loewy length 2 then $S^{\widetilde{\lambda}} \uparrow^{\mathcal{B}}$ has Loewy length at most 4.
		\item If $S^{\widetilde{\lambda}}$ has Loewy length 3 then $S^{\widetilde{\lambda}} \uparrow^{\mathcal{B}}$ has Loewy length at most 5.
	\end{enumerate}  
\end{lem}   
\begin{proof}
(1) If $S^{\widetilde{\lambda}}$ is simple then $S^{\widetilde{\lambda}} \cong D^{\widetilde{\lambda}^R}$, where $\widetilde{\lambda}^R$ is the $p$-regularization of the partition $\widetilde{\lambda}$.  By Proposition \ref{prop:MT}(2), $S^{\widetilde{\lambda}} \uparrow^{\mathcal{B}} \cong D^{\widetilde{\lambda}^R} \uparrow^{\mathcal{B}}$ has Loewy length 3.  Now (2) follows directly from Proposition \ref{prop:MT}(4).  Finally, (3) follows from Proposition \ref{prop:MT}(3),(4) and Proposition \ref{prop:bipart}.     
\end{proof}

We now cite a result that places an upper bound on the Loewy of the Specht module that lie in a defect 2 block of the symmetric group.     
\begin{prop}[\cite{CT}, Proposition 6.2] \label{prop:w2}
Let $B$ be a defect 2 block of $F\Sigma_n$, where $F$ is a field of characteristic $p\geq 3$.  Further, assume $\lambda$ is a partition in $B$.  Then the Loewy length of $S^{\lambda}$ is bounded above by 3.  Moreover, $S^{\lambda}$ has Loewy length 3 if and only if 
$\lambda$ is both $p$-regular and $p$-restricted.  
\end{prop}

The rules for restricting Specht modules to defect 2 blocks of $F\Sigma_{3p-1}$ discussed above, the work of Carter and Payne in \cite{CP}, and the modular branching rules of Kleshchev give the following proposition.

\begin{prop} \label{prop:CP}
Assume that the $\langle 3^p \rangle$ display for the partition $\lambda$ has a removable bead on runner $2 \leq i \leq p$.  
	\begin{enumerate}
			\item If $\lambda$ and $\sigma_i(\lambda)$ are $p$-regular then $\text{Hom}_{\Sigma_{3p}} (S^{\lambda}, S^{\sigma_i(\lambda)}) \neq 0$.  Furthermore,  we have $[S^{\sigma_i(\lambda)}: D^{\lambda}] =1$.
			\item There is a short exact sequence
			$$0 \rightarrow S^{\sigma_i(\lambda)} \rightarrow \left( S^{\lambda} \downarrow_{B_i}\right)\uparrow^{\mathcal{B}} \rightarrow S^{\lambda} \rightarrow 0.$$
	\end{enumerate}      
\end{prop}

\subsection{Abacus notation for the block $B_i$}
At times we will need to analyze the Ext$^1$-quiver of the block $B_i$.  Suppose that the $\langle3^p \rangle$ display for $\lambda$ has a removable bead on runner $2 \leq i \leq p$ at position $m$.  Pushing the bead from position $m$ to position $m-1$ gives an abacus display (for the partition $\tilde{\lambda}$ in $B_i$) which has 3 beads on the first $i-2$ runners, 4 beads on runner $i-1$, 2 beads on runner $i$, and 3 beads on the remaining $p-i$ runners.  So we may represent $\tilde{\lambda}$ in $\langle 3^{i-2}, 4,2, 3^{p-i} \rangle$ notation.  In fact, all partitions in $B_i$ can be represented in this notation.  However, the Ext$^1$-quiver of the block $B_i$ constructed in \cite{MR} has partitions represented in $\langle 2, 3^{p-i}, 2^{i-2},3 \rangle$ notation.  The authors provide a way to get from the 
$\langle 3^{i-2}, 4,2, 3^{p-i} \rangle$ notation to the $\langle 2, 3^{p-i}, 2^{i-2},3 \rangle$ notation.  Lemma 3.10 of \cite{MR} gives the following proposition.

\begin{prop} \label{prop:images}	
The partitions $\Theta_s(\lambda)$ ($2 \leq s \leq p$) listed below are expressed in $\langle 2, 3^{p-s},2^{s-2},3 \rangle$ notation.
\vskip .1in
\noindent (1) For $3 \leq i \leq p-1$,
	\begin{enumerate}
		\item [(a)] $\Theta_{2} (\langle i,1 \rangle) = \langle i-1 \rangle$;
		\item [(b)] $\Theta_{i}(\langle i,1 \rangle) = \langle p,p-i+2 \rangle$;
		\item [(c)] $\Theta_{i+1}(\langle i,1 \rangle ) = \langle p,p-i+1 \rangle$.
	\end{enumerate}
\noindent (2) For $2 \leq j < i \leq p-1$, $i - j \geq 2$,
	\begin{enumerate}
		\item [(a)] $\Theta_{j}(\langle i,j \rangle ) = \langle i-j+1 \rangle$;
		\item [(b)] $\Theta_{j+1}(\langle i,j \rangle ) = \langle i-j \rangle$;
		\item [(c)] $\Theta_{i}(\langle i,j \rangle)  = \langle p, p -(i-j)+1 \rangle$;
		\item [(d)] $\Theta_{i+1}(\langle i,j \rangle) = \langle p , p -(i-j) \rangle$.
	\end{enumerate}
\end{prop} 
We only represent $\Theta_s(\lambda)$ in  $\langle 2, 3^{p-s},2^{s-2},3 \rangle$ notation when we need to analyze the Ext$^1$-quiver of $B_s$.  In all other instances, $\Theta_s(\lambda)$ is assumed to be in $\langle 3^{s-2}, 4,2, 3^{p-s} \rangle$ notation.

\section{Partitions in the principal block of $F\Sigma_{3p}$} \label{section:partitions}

As discussed in the previous section, we may associate each partition in $\mathcal{B}$ with its $\langle 3^p \rangle$ display.  At times we refer to a 
$\langle 3^p \rangle$ display as the partition it represents and vice versa.  After analyzing the $\langle 3^p \rangle$ displays, we have Proposition \ref{prop:3p}. 
\begin{prop} \label{prop:3p} The following statements are true.  
	\begin{enumerate}
		\item $ \langle i,j \rangle $ is both $p$-regular and $p$-restricted if and only if $i<j \leq p-1$.
		\item $\langle i,i,j \rangle$ is both  $p$-regular and $p$-restricted if and only if $2 \leq j < i$.
		\item Suppose $i> j > k$.  Then $ \langle i,j,k \rangle $ is both $p$-regular and $p$-restricted if and only if  $(i,j,k) \notin \{(3,2,1), (p,p-1,p-2) \}$.
		\item A partition in $\mathcal{B}$ is both $p$-regular and $p$-restricted if and only if it has $\langle 3^p \rangle$ display of the form (1), (2), or (3).  
		\item $\lambda = \lambda'$ if and only if $\lambda$ has $\langle 3^p \rangle$ representation of the form $\langle \frac{p+1}{2}, \frac{p+1}{2} \rangle$ or $ \langle p-m, \frac{p+1}{2}, m  \rangle $ for some 
		$1\leq m \leq \frac{p-1}{2}$. 
		\item $\lambda$ is neither $p$-regular nor $p$-restricted if and only if $\lambda$ has $\langle 3^p \rangle$ display of the form $\langle i,i \rangle$ for some $1 \leq i \leq p$.  In this case, $\lambda = (p+i, 1^{2p-i})$.    
		\item $\lambda$ is a hook if and only if $\lambda$ has $\langle 3^p \rangle$ display of the form $\langle i \rangle$, $\langle i,i \rangle$, or $\langle i,i,i \rangle$ for some $1 \leq i \leq p$.
		\item $\lambda$ is $p$-regular, not $p$-restricted, and not a hook if and only if $\lambda$ has $\langle 3^{p} \rangle$ display of the form:
			\begin{enumerate}
				\item  $\langle i,p \rangle$ for $1 \leq i \leq p-1$,
				\item $\langle i,j \rangle$ for $1 \leq j < i \leq p$, or
				\item $\langle p,p-1,p-2 \rangle$.
			\end{enumerate}
	\end{enumerate}
\end{prop}

When working over a field of characteristic $p$ one often wishes to know which of the ordinary irreducible representations of the  symmetric group remain irreducible modulo $p$.  By Proposition \ref{prop:irredSpecht}, we know that the $p$-JM partitions label such representations.  Presently we determine which partitions in $\mathcal{B}$ are $p$-JM partitions.  We apply the test of Proposition \ref{prop:FJM}.

\begin{rem} \label{rem:irred}
Any partition $\lambda$ in $\mathcal{B}$ satisfies conditions (2) and (3) of Proposition \ref{prop:FJM} since $ \max( |\mu^{(1)}|, |\mu^{(p)}|)\leq 3$ and $p\geq 5$.  Suppose $\delta$ is a partition of $n < p$.  Then $\delta$ is both $p$-regular and $p$-restricted since its Young diagram has less than $p$ boxes.  Furthermore, for any node $(i,j)$ of $\delta$ we have 
$\nu_p(h_{\delta}(i,j)) = 0$.  Hence $\delta$ is also a $p$-JM partition.                
\end{rem}

We now show the only irreducible Specht modules in $\mathcal{B}$ are $S^{(3p)}$ and $S^{(1^{3p})}$.

\begin{thm} \label{thm:jm}
Let $\lambda$ be a partition in $\mathcal{B}$.  Then $\lambda$ is a $p$-JM partition if and only if 
$\lambda \in \{(3p), (1^{3p}) \}$.  
\end{thm}

\begin{proof}  
Obviously, if $\lambda \in \{(3p), (1^{3p})\}$ then it is a $p$-JM partition.  Suppose $\lambda$ is a $p$-JM-partition in $\mathcal{B}$.  Fix the $\langle 3^p \rangle$ abacus display for $\lambda$.  Pushing all beads up as far as possible on the $\langle 3^p \rangle$ display for $\lambda$ gives display $\Gamma$ (for the empty partition) and the left-to-right $p$-quotient $[\lambda^{(1)}, \lambda^{(2)}, \ldots \lambda^{(p)}]$.  It is clear that $q_i = 3p+i$ for all $1 \leq i \leq p$.  Thus the left-to-right and reordered $p$-quotients of $\lambda$ are the same and $B^k_{\ell} = 0$ for any $1 \leq k < \ell \leq p$.  By Remark \ref{rem:irred}, we need to only consider conditions (1) and (4) Proposition \ref{prop:FJM}.  Since $\lambda$ is a $p$-JM partition, $\lambda^{(j)} = \emptyset$ for $2 \leq j \leq p-1$ by Proposition \ref{prop:FJM}(1).  So the $p$-quotient becomes $[\lambda^{(1)}, \emptyset, \ldots, \emptyset, \lambda^{(p)}]$, and we have $|\lambda^{(1)}| + |\lambda^{(p)}| = 3$.  Now, since $B^k_{\ell} = 0$ for any $k < \ell$, condition (4) of Proposition \ref{prop:FJM} becomes 
\begin{equation} \label{eq:lam}
(\lambda^{(k)})_1 + ((\lambda^{(\ell)})')_1 \leq  1 \text{ for all }k < \ell.
\end{equation}  
We claim this implies either 
$\lambda^{(1)} = \emptyset$ and $\lambda^{(p)} = (3)$ or $\lambda^{(1)} = (1^3)$ and $\lambda^{(p)} = \emptyset$ [in which case $\lambda =(3p)$ or $\lambda = (1^{3p})$].  We may choose $ k =1$ and $\ell = p$ in Equation \ref{eq:lam}.  Then $(\lambda^{(1)})_1 + ((\lambda^{p})')_1 \leq  1$.  Since 
$|\lambda^{(1)}| + |\lambda^{(p)}| = 3$, $(\lambda^{(1)})_1$ and $((\lambda^{(p)})')_1$ cannot both be zero.  So either $(\lambda^{(1)})_1=0$ and $((\lambda^{(p)})')_1 = 1$ or $(\lambda^{(1)})_1=1$ and $((\lambda^{(p)})')_1 = 0$.  The first case implies $\lambda^{(1)} = \emptyset$ and $\lambda^{(p)} = (3)$ and the second case implies $\lambda^{(1)} = (1^3)$ and $\lambda^{(p)} = \emptyset$.  We conclude that if $\lambda$ in $\mathcal{B}$ is a $p$-JM partition then $\lambda \in \{(3p), (1^{3p}) \}$.                                     
\end{proof}

This section is concluded with a lemma that allows one to relate certain Specht modules in $\mathcal{B}$ to Specht modules in defect 2 blocks of $F\Sigma_{3p-1}$.  Application of the lemma is seen in the Section 5.

\begin{lem} \label{lem:remove}
Let $\lambda$ be a partition in the principal block of $F\Sigma_{3p}$.  If $\lambda$ is both $p$-regular and $p$-restricted then $\lambda$ has a removable node $A$ such that the following hold:   
	\begin{enumerate}
		\item [(i)] $\lambda_A$ has $p$-weight 2, and 
		\item [(ii)] $\lambda_A$ is both $p$-regular and $p$-restricted.
	\end{enumerate}
Moreover, if $\lambda \neq \langle1,2 \rangle $ then it has at least one normal node that satisfies both (i) and (ii).
\end{lem}

\begin{proof}
Let $\lambda$ be a partition in the principal block of $F\Sigma_{3p}$ which is both $p$-regular and $p$-restricted.  Then the $\langle 3^p \rangle$ display for $\lambda$ must be of the form (1), (2), or (3) of Proposition \ref{prop:3p}.

\vskip .1in
\noindent \textbf{Case 1.}  $\lambda $ has $\langle 3^p \rangle$ display of form $\langle i,j \rangle$, $1\leq i < j \leq p-1$.  Assume $i \neq 1$.  Push the bead at position $4p+i$ to position $4p+i-1$ (see Figure \ref{fig:case1}(A)).  This gives an abacus display for a weight 2 partition $\lambda_A$, which is clearly $p$-regular.  Notice that the longest string of spaces between beads on the abacus display for $\lambda_A$ has length at most $p-1$  (which occurs when $j = i+1$).  Hence $\lambda_A$ is $p$-restricted.  Finally, $A$ is normal since it is the top removable node of $\lambda$ (the last occupied position on the display is $4p+i$).  Now consider the display of the form $\langle1,j \rangle$, $2 \leq j \leq p-1$.  Push the bead at position $3p+j$ (see Figure \ref{fig:case1}(B)) to position $3p+j-1$.  Again, we obtain a partition 
$\lambda_A$ that satisfies (i) and (ii).  If $j \geq 3$ then there is no addable bead on runner $j-1 \geq 2$ after position $3p+j-1$.  Hence $A$ is normal for $j \geq 3$.  Finally, when $j=2$ we see exactly two removable nodes (corresponding to beads at positions $2p+3$ and $3p+2$) that satisfy conclusions (i) and (ii); however, neither of these nodes is normal.

\begin{figure}[h!]
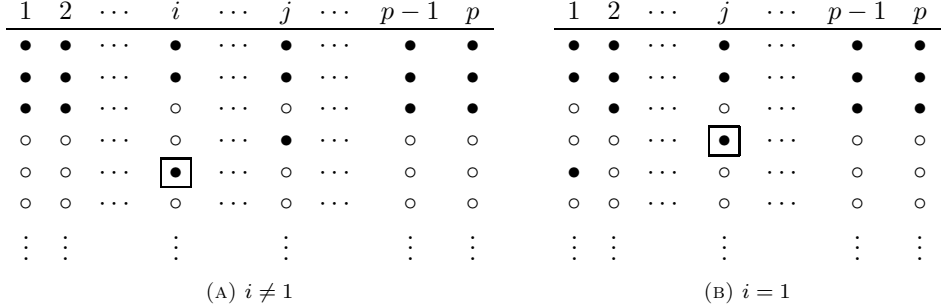
 
\centering
\subfloat[$i\neq 1$]{
$
\begin{array}{ccccccccc}
1        &    2     &     \cdots   &      i                        &      \cdots       &       j          &      \cdots     &    p-1     &   p \\
\hline
\bullet &\bullet &    \cdots   & \bullet                     &     \cdots        &    \bullet       &    \cdots     &   \bullet  &  \bullet \\ 
\bullet &\bullet &    \cdots   & \bullet                     &     \cdots        &    \bullet       &    \cdots     &   \bullet  &  \bullet \\
\bullet &\bullet &    \cdots   & \circ                        &     \cdots        &    \circ          &    \cdots     &   \bullet  &  \bullet \\
\circ   &   \circ  &    \cdots   & \circ                        &     \cdots        &    \bullet          &    \cdots     &  \circ  &  \circ \\
\circ   &  \circ  &    \cdots   & \boxed{\bullet}        &     \cdots        &    \circ          &    \cdots     &   \circ  &  \circ\\ 
\circ   &  \circ  &    \cdots   & \circ                       &     \cdots        &    \circ          &    \cdots     &   \circ  &  \circ\\
\vdots &\vdots&                & \vdots                     &                        &\vdots          &                &    \vdots   & \vdots
\end{array}
$
}
\quad \ \ 
\subfloat[$i =1$]{
$\begin{array}{ccccccc}
  1               &     2      &     \cdots   & 	   j                     &  \cdots   &       p-1    & p \\
\hline
 \bullet        & \bullet   &      \cdots & \bullet                &   \cdots  &    \bullet  &  \bullet \\
 \bullet        & \bullet   &      \cdots & \bullet                &   \cdots  &    \bullet  &  \bullet \\
 \circ           & \bullet   &      \cdots &   \circ                 &   \cdots  &    \bullet  &  \bullet \\
 \circ           &      \circ   &      \cdots & \boxed{\bullet} &   \cdots  &    \circ    &  \circ \\
 \bullet        &      \circ   &      \cdots &   \circ               &   \cdots  &    \circ    &  \circ \\
  \circ          &      \circ   &      \cdots &   \circ               &   \cdots  &    \circ    &  \circ \\
  \vdots       &   \vdots    &                  &\vdots            &             &  \vdots   & \vdots 
\end{array}
$
}
\caption{$\langle i,j \rangle, 1 \leq i < j \leq p-1$}
\label{fig:case1}
\end{figure}



\vskip .1in
\noindent \textbf{Case 2.}  $\lambda $ has $\langle 3^p \rangle$ display of form $\langle i,i,j \rangle$, $2\leq j < i \leq p$.  First assume $i \neq j+1$.  Push the bead at position $3p+i$ to position $3p+i-1$ (See Figure \ref{fig:case2}(A)).  The resulting abacus display represents a partition that satisfies both (i) and (ii), and since position $3p+i$ is the last occupied position the corresponding removable node is normal.  If $i = j+1$ then push the bead at position $3p+j$ to position $3p+j-1$ (see Figure \ref{fig:case2}(B)).  The resulting abacus display represents a partition that satisfies both (i) and (ii).  Now since runner $j-1$ has no addable bead at a position after position $3p+j-1$ the corresponding removable node is normal.   

\begin{figure}[ht!]
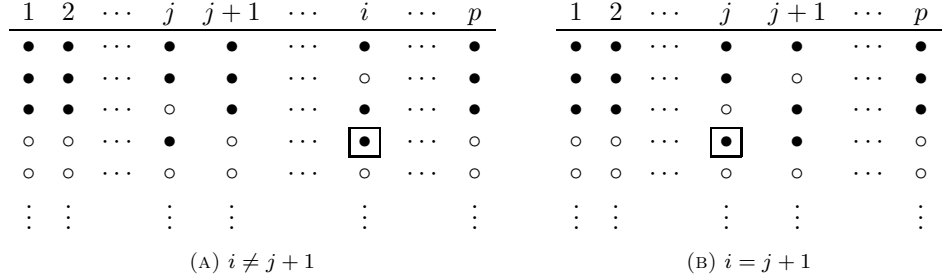

\centering
\subfloat[$i\neq j+1$]{
$
\begin{array}{ccccccccc}
1        &    2     &     \cdots   &      j                       &      j+1       &       \cdots      &      i                 &    \cdots     &   p \\
\hline
\bullet &\bullet &    \cdots   & \bullet                     &    \bullet        &    \cdots     &  \bullet           &   \cdots      &  \bullet \\ 
\bullet &\bullet &    \cdots   & \bullet                     &    \bullet        &    \cdots    &       \circ          &   \cdots      &  \bullet \\
\bullet &\bullet &    \cdots   & \circ                        &    \bullet         &    \cdots   &   \bullet           &   \cdots      &  \bullet \\
\circ   &   \circ  &    \cdots   & \bullet                      &       \circ        &    \cdots    & \boxed{\bullet}&  \cdots       &  \circ \\
\circ   &  \circ  &    \cdots   & \circ                         &       \circ        &    \cdots       &    \circ           &   \cdots  &  \circ\\ 
\vdots &\vdots&                & \vdots                      &  \vdots            &                   &     \vdots        &              & \vdots
\end{array}
$
}
\quad \ \ 
\subfloat[$i =j+1$]{
$\begin{array}{ccccccc}
  1               &     2      &     \cdots   & 	       j                &    j+1   &      \cdots      & p \\
\hline   
  \bullet       & \bullet   &  \cdots     &   \bullet             & \bullet &    \cdots         & \bullet \\
  \bullet       &  \bullet  &  \cdots     &  \bullet              & \circ    &    \cdots         & \bullet \\
  \bullet       &  \bullet  &  \cdots     &  \circ                 & \bullet &    \cdots         & \bullet \\
 \circ           &  \circ     &  \cdots     & \boxed{\bullet}  & \bullet &     \cdots         & \circ \\
 \circ           &  \circ     &  \cdots     &  \circ                 & \circ    &     \cdots         & \circ \\
 \vdots        & \vdots   &                & \vdots                &\vdots &                        & \vdots \\
\end{array}
$
}
\caption{$\langle i,i,j \rangle$, $2 \leq j < i \leq p$}
\label{fig:case2}
\end{figure}


\vskip .1in
\noindent \textbf{Case 3.}  $\lambda $ has $\langle 3^p \rangle$ display of form $\langle i,j,k \rangle$, $i>j>k$ and $(i,j,k) \notin \{(3,2,1),(p,p-1,p-2) \}$.  We first assume $k=1$.  If $j = 2$ then $i \geq 4$.  Push the bead at position $3p+i$ to position $3p+i-1$ (see Figure \ref{fig:case3a}(A)).  This gives a display for a weight 2 partition that is both $p$-regular and $p$-restricted.  The node that corresponds to the bead at position $3p+i$ is normal since runner $i-1$ has no addable bead at a position after position $3p+i-1$.  If $j \geq 3$ we push the bead at position $3p+j$ to position $3p+j-1$.  Again we obtain a partition that satisfies (i) and (ii) of the lemma.  The node that corresponds to the bead at position $3p+j$ is normal since runner $j-1$ has no addable bead at a position after position $3p+j-1$.  

\begin{figure}[h!]
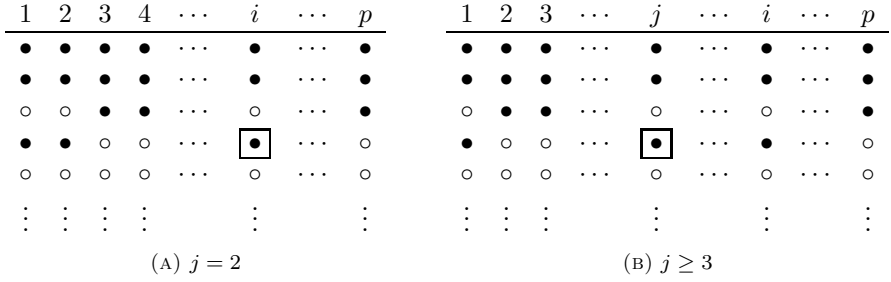
 
\centering
\subfloat[$j=2$]{
$
\begin{array}{cccccccc}
1            &     2    &   3        &   4            &   \cdots &    i                    &   \cdots   &  p\\
\hline 
\bullet   & \bullet & \bullet   &  \bullet    & \cdots    & \bullet              &   \cdots    & \bullet \\
\bullet   & \bullet & \bullet   &  \bullet    & \cdots    & \bullet              &   \cdots    & \bullet \\
\circ     &  \circ   & \bullet   &  \bullet     & \cdots    & \circ                 &   \cdots    & \bullet \\
\bullet  & \bullet &  \circ      & \circ         & \cdots   & \boxed{\bullet}  & \cdots       & \circ \\
\circ     &  \circ  & \circ       & \circ         & \cdots   & \circ                  & \cdots      &\circ \\
\vdots  &\vdots &\vdots     & \vdots      &              & \vdots               &                & \vdots
\end{array}
$
}
\quad \ \
\subfloat[$j \geq 3$]{
$
\begin{array}{ccccccccc}
1        &    2      &   3       &    \cdots     &           j             &     \cdots     &      i       &  \cdots    &      p \\
\hline
\bullet & \bullet & \bullet &    \cdots     &    \bullet           &    \cdots      & \bullet    &   \cdots  &  \bullet\\
\bullet & \bullet & \bullet &    \cdots     &    \bullet           &    \cdots      & \bullet    &   \cdots  &  \bullet \\
\circ    & \bullet & \bullet &    \cdots     &    \circ              &    \cdots      & \circ       &   \cdots  &  \bullet \\
\bullet & \circ    &  \circ   &    \cdots     & \boxed{\bullet} &  \cdots       & \bullet     & \cdots    & \circ  \\
\circ    &  \circ   &   \circ  &    \cdots    &     \circ              &  \cdots      &   \circ      & \cdots     & \circ \\
\vdots & \vdots & \vdots &                 &  \vdots               &                 &  \vdots    &                & \vdots     
\end{array}
$
}
\caption{$\langle i,j,1 \rangle$}
\label{fig:case3a}
\end{figure}

\noindent Next assume $k \geq 2$.  Push the bead at position $3p+k$ to position $3p+k-1$ to obtain a partition satisfying (i) and (ii) (see Figure \ref{fig:case3b}).  The node that corresponds to the bead at position $3p+k$ is normal since runner $k-1$ has no addable bead at a position after position $3p+k-1$.

\begin{figure}[h!]
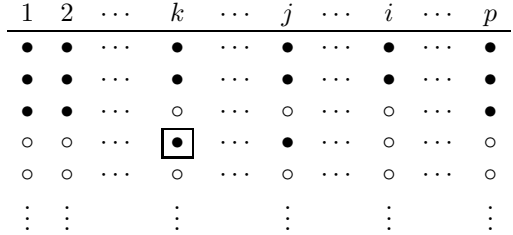
 
\centering
$
\begin{array}{cccccccccc} 
     1         &           2     &   \cdots   &               k               &    \cdots     &       j        &    \cdots     &      i       &     \cdots   &   p \\
\hline
\bullet      &      \bullet   &  \cdots    &            \bullet         &   \cdots        &   \bullet  &  \cdots       &   \bullet  &   \cdots    & \bullet \\
\bullet      &      \bullet   &  \cdots    &            \bullet         &   \cdots        &   \bullet  &  \cdots       &   \bullet  &   \cdots    & \bullet \\
\bullet      &      \bullet   &  \cdots    &            \circ            &   \cdots        &   \circ     &  \cdots       &   \circ     &   \cdots    & \bullet  \\
\circ         &     \circ       & \cdots     &    \boxed{\bullet}    & \cdots          & \bullet    & \cdots        & \circ       &   \cdots    & \circ     \\
\circ         &     \circ       & \cdots     &           \circ             & \cdots          &  \circ       & \cdots       & \circ       &   \cdots    & \circ    \\ \vdots      &  \vdots       &               &       \vdots              &                     &  \vdots    &                  & \vdots    &                & \vdots
\end{array}
$
\caption{$\langle i,j,k \rangle, \ k \geq 2$}
\label{fig:case3b}
\end{figure}

It follows from the proofs of Cases (1)-(3) that $\lambda \neq \langle1,2 \rangle$ has at least one normal node that satisfies both (i) and (ii).         
\end{proof}

\begin{cor} \label{cor:I2}
Assume $\lambda$ is a partition in $\mathcal{B}$ that is both $p$-regular and $p$-restricted and $\lambda \neq \langle1,2 \rangle$.  Then $\lambda \in \Lambda_i$ for some $2 \leq i \leq p$.  
\end{cor}

\section{Irreducible Specht modules in the defect 2 block $B_i$ of $F\Sigma_{3p-1}$} \label{section:irreducibles}

In this section we determine the Specht modules in the block $B_i$ ($2 \leq i \leq p$) that remain irreducible over the field $F$.  We then induce these modules back up to the block $\mathcal{B}$ to see which Specht modules in $\mathcal{B}$ have an irreducible restriction to the block $B_i$.  For $2 \leq i \leq p$, define $X_i  : = \{\tilde{\lambda} \in B_i \  | \ S^{\tilde{\lambda}} \text{ is irreducible}  \}$.  After applying Proposition \ref{prop:FJM} to each block $B_i$ we have the following lemma. 

\begin{lem} \label{lem:simpledefect2}The following is a complete list of partitions in the block $B_i$ ($2 \leq i \leq p$) represented in $\langle 3^{i-2}, 4,2, 3^{p-i} \rangle$ notation that label irreducible Specht modules.   
	\begin{enumerate}
		\item $X_i = \{\langle i,i \rangle, \langle i-1 \rangle, \langle i , i-1 \rangle  \}$ for $3 \leq i \leq p-1$;
		\item $X_p = \{ \langle p,p \rangle, \langle p,p-1 \rangle, \langle p-1 \rangle, \langle p-1, p-1 \rangle \}$;
		\item $X_2 = \{\langle 2,2 \rangle, \langle 2 \rangle, \langle 1 \rangle, \langle 2,1 \rangle \}$
	\end{enumerate}
\end{lem}

\begin{proof}
We prove the result for the block $B_2$.  The blocks $B_3, \ldots, B_p$ will have a similar calculation.  The $p$-core $(1^{3p-1})$ for the block $B_2$ can be represented on an abacus with 4 beads on runner 1, 2 beads on runner 2, and 3 beads on runners $3, \ldots, p$.  Hence 
$q_1 = 2p+2$, $q_k = 3p+k+1$ for $2 \leq k \leq p-1$, and $q_p = 4p+1$.  So we renumber the runners from left to right as 
$p, 1, 2, 3, \ldots p-1$.  Notice that in this renumbering $1 \mapsto p$ and $2 \mapsto 1$.  The block $B_2$ has pyramid 
$B^1_{\ell} = 1$ if $2 \leq \ell \leq p$ and $B^{k}_{\ell} = 0$ if $2 \leq k < \ell \leq p$.  Now suppose $\lambda$ is a $p$-JM partition in $B_2$.  Since $\lambda$ is a $p$-JM partition, the reordered $p$-quotient of $\lambda$ has the form 
$[\mu^{(1)} , \emptyset , \ldots,\emptyset, \mu^{(p)}]$ by Proposition \ref{prop:FJM}(1).  Note that $\lambda$ has $p$-weight 2, so $\mu^{(1)}$ and $\mu^{(p)}$ satisfy conditions (2) and (3) of Proposition \ref{prop:FJM}, respectively.  So we consider condition (4), and   observe that $|\mu^{(1)} | + | \mu^{(p)}| = 2$.  We consider three cases.      

\vskip .1in
\noindent \textbf{Case 1.} Assume $|\mu^{(1)} | = 2$ and $|\mu^{(p)}| = 0$.  First suppose $\mu^{(1)} = (2)$.  Then for $1 \leq k < \ell \leq p$ we have the following. 
$$\left(\mu^{(k)}\right)_{1} + \left( (\mu^{(\ell)})' \right)_{1}  = \begin{cases}  2  \text{ if $k =1$} \\ 0 \text{ if $k \neq 1$}  \end{cases}$$  
Since $B^{1}_{\ell} =1$ for all $2 \leq \ell \leq p$ the reordered $p$-quotient of $\lambda$ does not contradict condition (4).  Now $\lambda$ has left-to-right $p$-quotient $[\emptyset, (2), \emptyset, \ldots, \emptyset]$.  So $\lambda = \langle 2 \rangle$.  If $\mu^{(1)} = (1^2)$ then one can verify that the reordered $p$-quotient satisfies condition (4).  Now the left-to-right $p$-quotient is $[\emptyset, (1^2), \emptyset, \ldots, \emptyset]$.  Hence $\lambda = \langle 2,2 \rangle$.

\vskip .1in
\noindent \textbf{Case 2.} Assume $|\mu^{(1)} | = 0$ and $|\mu^{(p)}| = 2$.  If $\mu^{(p)} = (2)$ then we have the following.
$$\left(\mu^{(k)}\right)_{1} + \left( (\mu^{(\ell)})' \right)_{1}  = \begin{cases}  1  \text{ if $k =p$} \\ 0 \text{ if $k \neq p$}  \end{cases}$$
It is easily seen that the reordered $p$-quotient satisfies condition (4).  The left-to-right $p$-quotient is 
$[(2), \emptyset, \ldots, \emptyset]$, so $\lambda = \langle 1 \rangle$.  Now suppose $\mu^{(p)} = (1^2)$.  Observe that  
$\left(\mu^{(2)}\right)_1 + \left( ( \mu^{(p)})' \right)_{1} = 0 + 2 = 2$ and $B^{2}_{p} + 1 = 0 + 1 = 1$.  So this $p$-quotient does not satisfy condition (4).

\vskip .1in
\noindent \textbf{Case 3.}  Assume $|\mu^{(1)}| = 1 = |\mu^{(p)}|$.  One may easily verify that this reordered $p$-quotient satisfies condition (4).  The left-to-right $p$-quotient is $[(1), (1), \emptyset, \ldots, \emptyset]$.  So $\lambda = \langle 2,1\rangle$. 

\vskip .1in
\noindent Therefore $X_2 = \{  \langle 2,2 \rangle, \langle 2 \rangle, \langle 1 \rangle,  \langle 2,1\rangle \}$.                   
\end{proof}

Lemma \ref{lem:simpledefect2} and the modular branching rules give us the next proposition.

\begin{prop} \label{prop:spechtsimpleinduce}
Inducing the irreducible Specht modules in the block $B_i$  ($2 \leq i \leq p$) up to $\mathcal{B}$ produces the following Specht module filtrations.  Furthermore, if a Specht module in $\mathcal{B}$ has an irreducible restriction to some block $B_i$ ($2 \leq i \leq p$) then it must be seen in one of these filtrations. 
\begin{enumerate}
	\item For $i =2$
	$$\begin{matrix} S^{\langle 2,2 \rangle}\uparrow^{\mathcal{B}} \sim \begin{matrix} S^{\langle 2,2,2 \rangle} \\ S^{\langle 1,1,1 \rangle}  \end{matrix}, &
	S^{\langle 2 \rangle}\uparrow^{\mathcal{B}} \sim \begin{matrix} S^{\langle 2,2,1 \rangle} \\ S^{\langle 2,1,1 \rangle}  \end{matrix},\\ \\
	S^{\langle 1 \rangle}\uparrow^{\mathcal{B}} \sim \begin{matrix} S^{\langle 2 \rangle} \\ S^{\langle 1 \rangle}  
	\end{matrix}, &
	S^{\langle 2,1 \rangle}\uparrow^{\mathcal{B}} \sim \begin{matrix} S^{\langle 2,2 \rangle} \\ S^{\langle 1,1 \rangle}  \end{matrix}
\end{matrix}$$ 
	\item For $3 \leq i \leq p-1$
	$$S^{\langle i,i\rangle}\uparrow^{\mathcal{B}} \sim \begin{matrix} S^{\langle i,i,i \rangle} \\ S^{\langle i-1,i-1.i-1 \rangle} \end{matrix}, \quad
	S^{\langle i-1 \rangle} \uparrow^{\mathcal{B}} \sim \begin{matrix} S^{\langle i \rangle}\\ S^{\langle i-1 \rangle}  \end{matrix}, \quad
	S^{\langle i,i-1 \rangle} \uparrow^{\mathcal{B}} \sim \begin{matrix} S^{\langle i,i \rangle}\\ S^{\langle i-1,i-1 \rangle}  \end{matrix}$$
	\item For $i=p$
	$$\begin{matrix} S^{\langle p,p \rangle}\uparrow^{\mathcal{B}} \sim \begin{matrix} S^{\langle p,p,p \rangle} \\ S^{\langle p-1,p-1,p-1 \rangle}  \end{matrix}, &
	S^{\langle p,p-1 \rangle}\uparrow^{\mathcal{B}} \sim \begin{matrix} S^{\langle p,p \rangle} \\ S^{\langle p-1,p-1 \rangle}  \end{matrix},\\ \\
	S^{\langle p-1 \rangle}\uparrow^{\mathcal{B}} \sim \begin{matrix} S^{\langle p \rangle} \\ S^{\langle p-1 \rangle}  
	\end{matrix}, &
	S^{\langle p-1,p-1 \rangle}\uparrow^{\mathcal{B}} \sim \begin{matrix} S^{\langle p,p-1 \rangle} \\ S^{\langle p-1,p \rangle}  \end{matrix}
\end{matrix}$$ 
\end{enumerate}
\end{prop}



\section{Loewy lengths and some loewy structures} \label{section:loewy}
In this section we study the Loewy length of $S^{\lambda}$ where $\lambda$ is a partition in the principal block of $\Sigma_{3p}$.  We see that the maximum possible Loewy length of a Specht module in the block $\mathcal{B}$ is 4.  Compare this with the result of Chuang and Tan given in Proposition \ref{prop:w2}.  In \cite{MT}, Martin and Tan prove that the projective indecomposable modules of principal blocks of $F\Sigma_{n}$ ($3p \leq n \leq 4p-1$) have common Loewy length 7.  Tan later generalizes this result to all weight 3 blocks (with $p\geq 5$) of $F\Sigma_{n}$ in \cite{Tan}.   This fact aids in finding a common Loewy length among the Specht modules in a given defect 3 block of $F\Sigma_{n}$ that correspond to partitions that are both $p$-regular and $p$-restricted.  After the upper bound on the Loewy lengths is established, we prove some Loewy structure results for the Specht modules in $\mathcal{B}$.


\subsection{The bound on the Loewy lengths}

We claim that all the Specht modules that lie in $\mathcal{B}$ have Loewy length at most 4.  We begin by showing that any Specht module $S^{\lambda}$ lying in a defect 3 block of $F\Sigma_{n}$, where $\lambda$ is a $p$-regular and $p$-restricted, has Loewy length 4.



\begin{thm} \label{thm:RR}
Suppose the partition $\lambda$ lies in a defect 3 block $B$ of $F\Sigma_{n}$.  If $\lambda$ is both $p$-regular and $p$-restricted then the Specht module $S^{\lambda}$ has Loewy length 4.
\end{thm}

\begin{proof}
Since $\lambda$ is 
$p$-restricted, $Y^{\lambda}$ is isomorphic to the projective cover of $P(D^{m(\lambda ')})$ by Proposition \ref{prop:Young}(1).  So $Y^{\lambda}$ has Loewy length 7.  Now, $[Y^{\lambda} : D^{\lambda}] = 1$ (see Proposition \ref{prop:Young}(2)) implies that $D^{\lambda}$ must lie in the fourth Loewy layer of $Y^{\lambda}$ since $Y^{\lambda}$ is self-dual and indecomposable.  We conclude that $S^{\lambda}$ has Loewy length 4.    
\end{proof}

\begin{cor}
Suppose the partition $\lambda$ lies in a defect 3 block of $F\Sigma_n$.  If $\lambda$ is $p$-regular and $p$-restricted then 
$$[\text{rad}^3Y^{\lambda}/\text{rad}^4Y^{\lambda}:D^{\lambda}] =1 \text{ and } [\text{rad}^iY^{\lambda}/\text{rad}^{i+1}Y^{\lambda}:D^{\lambda}]  = 0$$
for all $0\leq i\neq 3 \leq 6.$ 
\end{cor}


We just proved that the Specht modules in a defect 3 block which have corresponding partition that is $p$-regular and $p$-restricted all have common Loewy length 4.  We claim that all the Specht modules in $\mathcal{B}$ have Loewy length at most 4.  The idea behind the proof is to restrict to ``nice'' defect 2 blocks of $F\Sigma_{3p-1}$ (i.e., we avoid restricting to Specht modules having Loewy length 3 if possible) then apply Lemma \ref{lem:sinduce}(1) and (2).  We also note that if $\lambda$ is a hook partition of $n = 3p$ then $S^{\lambda}$ has Loewy length at most 2.  Now the only partitions in $\mathcal{B}$ that are neither $p$-regular nor $p$-restricted have $\langle 3^p \rangle$ display of the form 
$\langle i,i \rangle$ for some $1 \leq i \leq p$.  So the only partitions in $\mathcal{B}$ that are neither $p$-regular nor $p$-restricted are the hooks $\lambda(i) = (3p - (2p-i), 1^{2p-i}) = (p+i, 1^{2p-i})$, $1\leq i \leq p$.  Therefore it suffices to consider $\lambda$ that is $p$-regular, not $p$-restricted, and not a hook (by Propositions \ref{prop:james}, \ref{prop:peel}, and Theorem \ref{thm:RR}).  For such a partition $\lambda$, we analyze the restriction of $S^{\lambda}$ to a particular defect 2 block of $F\Sigma_{3p-1}$ and apply Proposition \ref{prop:w2} in next lemma.

\begin{lem} \label{lem:nicerest}
Let $\lambda$ be a partition in $\mathcal{B}$ that is $p$-regular, not $p$-restricted, and not a hook.  Then there exists $2 \leq i \leq p$ such that $S^{\lambda} \downarrow_{B_i} \neq 0$ has Loewy length at most 2, except for when $\lambda = \langle p,p-1,p-2 \rangle$.     
\end{lem}

\begin{proof}
Any partition in $\mathcal{B}$ that satisfies the given hypothesis must have $\langle 3^p \rangle$ display of the form:
	\begin{enumerate}
		\item $\langle i,p \rangle$ where $1\leq i \leq p-1$,
		\item $\langle i,j\rangle$ where $1 \leq j < i \leq p$, or
		\item $\langle p,p-1,p-2 \rangle$.
	\end{enumerate}
(1) First assume $i \geq 2$.  Push the removable bead at position $4p+i$ of $\langle i,p \rangle$ to position $4p+i-1$.  This gives a partition $\tilde{\lambda} = \langle i-1, p\rangle$ in $\langle3^{i-2}, 4, 2, , 3^{p-i} \rangle$ notation.  Clearly, $\tilde{\lambda}$ is $p$-regular.  For each $0 \leq k \leq p-1$, the position $2p+i+k$ of $\langle i-1, p\rangle$ is unoccupied, and the last occupied position of $\langle i-1, p\rangle$ is $4p+i-1$.  So $\tilde{\lambda}'$ is $p$-singular.  By Proposition \ref{prop:w2}, the module 
$S^{\langle i,p \rangle}\downarrow_{B_i} \cong S^{\tilde{\lambda}}$ has Loewy length at most 2.  For the case $i=1$, observe that  $S^{\langle 1,p \rangle}\downarrow_{B_2} \cong S^{\langle 2,p \rangle}\downarrow_{B_2}$. 

\vskip .1in
\noindent (2) Assume $i \neq j+1$.  Push the removable bead at position $4p + i$ to position $4p+i-1$.  This gives a partition 
$\tilde{\lambda} = \langle i-1, j \rangle$ in $\langle 3^{i-2}, 4,2, 3^{p-i}\rangle$ notation that is $p$-regular but not $p$-restricted.  By Proposition \ref{prop:w2}, $S^{\langle i,j \rangle } \downarrow_{B_i} \cong S^{\tilde{\lambda}}$ has Loewy length at most 2. Now assume $i = j+1$.  First consider $j\geq 2$.  Pushing the removable bead at position $3p+j$ to position $3p+j-1$ gives the partition $\tilde{\lambda} = \langle j+1, j-1\rangle$ in $\langle 3^{j-2},4,2,3^{p-j} \rangle$ notation which is $p$-regular but not $p$-restricted.  By Proposition \ref{prop:w2}, $S^{ \langle j+1,j \rangle } \downarrow_{B_j} \cong S^{\tilde{\lambda}}$ has Loewy length at most 2.      
For the case $j=1$, observe that  $S^{\langle 2,1 \rangle } \downarrow_{B_3} \cong S^{\langle 3,1\rangle }\downarrow_{B_{3}}$ has Loewy length at most 2.   

\vskip .1in
\noindent (3) The partition $\lambda$ with display $\langle p,p-1,p-2 \rangle $ has exactly one removable node, $A$, of $p$-residue $p-3$.  By inspection, $\lambda_A$ is both $p$-regular and $p$-restricted.  Hence $S^{\langle p,p-1,p-2 \rangle} \downarrow_{B_{p-2}} \cong S^{\lambda_A}$ has Loewy length 3 by Proposition \ref{prop:w2}.        
\end{proof}

The exceptional partition, $\lambda = \langle p,p-1,p-2 \rangle$, of Lemma \ref{lem:nicerest} corresponds to a Specht module that has Loewy length 3.  To demonstrate this we show $S^{\lambda}$ embeds into the Specht module $S^{\langle p,p-1,p-3\rangle}$ which has Loewy length 4 and simple socle isomorphic to $D^{m(\langle p,p-1,p-3 \rangle ')} \cong D^{\langle p-2 \rangle}$. 

\begin{lem} \label{lem:help}
The Specht module $S^{\langle p,p-1,p-2 \rangle}$ is isomorphic to a submodule of $S^{\langle p,p-1,p-3 \rangle}$.  In particular, $S^{\langle p,p-1,p-2\rangle}$ has Loewy length 3 and simple socle isomorphic to $\text{soc} (S^{\langle p,p-1,p-3\rangle}) \cong D^{m(\langle p,p-1,p-3 \rangle ')} \cong D^{\langle p-2 \rangle}$.
\end{lem}

\begin{proof}
Let $\varphi \in  \text{Hom}_{\Sigma_{3p}}(S^{\langle p,p-1,p-2 \rangle} , S^{\langle p,p-1,p-3 \rangle})$ be a non-zero homomorphism (see Proposition \ref{prop:CP}).  Let $\lambda = \langle p,p-1,p-2\rangle $ and $\mu = \langle p,p-1,p-3 \rangle$.  By parity and Theorem \ref{thm:RR}, we know $D^{\lambda}$ is in the second Loewy layer of $S^{\mu}$ since $\lambda \neq \langle p-2 \rangle = m(\mu')  $.  Observe that $\lambda$ has exactly one removable node $A$ (which has residue $p-3$).  Hence $\lambda \in \Lambda_{p-2}$.  Furthermore, $S^{\lambda}\downarrow_{\Sigma_{3p-1}} \cong 
S^{\lambda}\downarrow_{B_{p-2}}$, $S^{\lambda} \downarrow_{B_{p-2}} \cong S^{\mu} \downarrow_{B_{p-2}} \cong S^{\lambda_A}$, and $\lambda > \mu$.  Clearly, $\lambda_A$ is both $p$-regular and $p$-restricted, so $S^{\lambda_A}$ has Loewy length 3 and simple socle $D^{m((\lambda_A)')} \cong D^{m(\mu')} \downarrow_{B_{p-2}}$.  Regard $\varphi$ as an $F\Sigma_{3p-1}$-module homomorphism.  Then $\varphi: S^{\lambda_A} \rightarrow S^{\lambda_A}$ is non-zero, so $\varphi$ is one-to-one since $[S^{\lambda_A} : D^{m((\lambda_A)')}] = 1$ (see Proposition \ref{prop:decnumb}).          
\end{proof}

The following lemmas help us prove that a Specht module in $\mathcal{B}$ corresponding to a partition that is $p$-regular, not $p$-restricted, and not a hook has Loewy length at most 3.  Ultimately, these lemmas help to prove the converse of Theorem \ref{thm:RR}.

\begin{lem} \label{lem:getit}
Let $\lambda \neq \langle p,p-1,p-2 \rangle $ be a $p$-regular, not $p$-restricted, and non-hook partition in $\mathcal{B}$.  Suppose there exists a $p$-regular partition $\delta$ in $\mathcal{B}$ such that the following hold:
	\begin{enumerate}
		\item $S^{\delta} \downarrow_{B_i} \cong S^{\lambda} \downarrow_{B_i} \neq 0$ (for some $2 \leq i \leq p$) with $\delta \rhd \lambda$ and 
		\item $S^{\lambda}  \downarrow_{B_i}$ has Loewy length at most 2.
	\end{enumerate}
	Then  $S^{\lambda}$ has Loewy length at most 3.  
\end{lem}

\begin{proof}
We have   
$D^{\delta} \downarrow_{B_i} \neq 0$ and $D^{\lambda} \downarrow_{B_i} = 0$.  Now 
$(S^{\delta} \downarrow_{B_i})\uparrow^{\mathcal{B}}$ has Loewy length at most 4 (by Lemma \ref{lem:sinduce}).  Observe that $(S^{\delta} \downarrow_{B_i})\uparrow^{\mathcal{B}}$ has a Specht filtration $S^{\delta} + S^{\lambda}$ with $S^{\delta}$ at the top.  By Frobenius reciprocity, $D^{\lambda}$ is not in the head of $(S^{\delta} \downarrow_{B_i})\uparrow^{\mathcal{B}}$ since $D^{\lambda} \downarrow_{B_i} =0$.  Hence 
$S^{\lambda} \subseteq \text{rad} \left( (S^{\delta} \downarrow_{B_i})\uparrow^{\mathcal{B}} \right)$.  So $S^{\lambda}$ has Loewy length strictly less than that of $(S^{\delta} \downarrow_{B_i}) \uparrow^{\mathcal{B}}$.  We conclude that $S^{\lambda}$ has Loewy length at most 3. 
\end{proof}

We are now ready to prove the converse of Theorem \ref{thm:RR}.  It suffices to prove the converse is true for $p$-regular partitions in $\mathcal{B}$ that are not $p$-restricted and also not hooks by Propositions \ref{prop:james} and \ref{prop:peel}.

\begin{thm} \label{thm:gotit}
Let $\lambda$ be a partition in $\mathcal{B}$ that is $p$-regular but not $p$-restricted.  Then $S^{\lambda}$ has Loewy length at most 3.
\end{thm}
\begin{proof}
 Let $\lambda$ be a partition in $\mathcal{B}$ that is $p$-regular but not $p$-restricted.  By Lemma \ref{lem:help}, we may assume $\lambda \neq \langle p,p-1,p-2 \rangle$.  Further, we may assume $\lambda$ is not a hook (by Proposition \ref{prop:peel}).  From Lemma \ref{lem:nicerest}, we can find a weight 2 block $B_i$ ($2 \leq i \leq p$) of $F\Sigma_{3p-1}$ such that $S^{\lambda}\downarrow_{B_i}$ has Loewy length at most 2, so $(S^{\lambda}\downarrow_{B_i}) \uparrow^{\mathcal{B}}$ has Loewy length at most 4.  In particular, we have the following pairs of restrictions.     
 	\begin{enumerate}
		\item $S^{ \langle i+1,p \rangle }\downarrow_{B_{i+1}} \cong S^{\langle i,p \rangle} \downarrow_{B_{i+1}} $ for $1 \leq i \leq p-2 $
		\item $S^{\langle j+2,j \rangle}\downarrow_{B_{j+2}}  \cong S^{\langle j+1,j\rangle } \downarrow_{B_{j+2}}$ for $1 \leq j \leq p-2$
		\item $S^{\langle i+1,j \rangle }\downarrow_{B_{i+1}} \cong S^{\langle i,j \rangle}\downarrow_{B_{i+1}}$ for $1 \leq j < i \leq p-1$, $i - j \geq 2$
		\item $S^{\langle p,j+1 \rangle} \downarrow_{B_{j+1}} \cong S^{\langle p, j \rangle} \downarrow_{B_{j+1}}$ for $1 \leq j \leq p-2$
	\end{enumerate}
   
\noindent Define the sets:
\begin{enumerate}
		\item [] $E_1 = \{\langle i,p \rangle: 1 \leq i \leq p-2 \}$, 
		\item [] $E_2 = \{ \langle j+1, j \rangle : 1 \leq j \leq p-2 \}$, 
		\item [] $E_3= \{\langle i,j \rangle : 1 \leq j < i \leq p-1, \ i-j \geq 2\},$ and 
		\item [] $E_4  = \{ \langle p,i\rangle : 1 \leq i \leq p-2\}$  
\end{enumerate}

\noindent If $\lambda \in E_1\cup E_2 \cup E_3 \cup E_4$ then $S^{\lambda}$ has Loewy length at most 3 by Lemma \ref{lem:getit} and the pairs of restrictions given above.  We just need to prove the theorem for the partitions $\langle p,p-1 \rangle $ and $\langle p-1, p \rangle$.  We have $S^{\langle p,p-1 \rangle} \downarrow_{B_p} \cong S^{\langle p-1,p \rangle} \downarrow_{B_p}$ is irreducible by Proposition \ref{prop:spechtsimpleinduce}(3).  Hence $S^{\langle p,p-1 \rangle}$ and $S^{\langle p-1,p \rangle}$ both have Loewy length at most 3 by Lemma \ref{lem:sinduce}(1) and since $\left(S^{\langle p,p-1\rangle} \downarrow_{B_p} \right)\uparrow^{\mathcal{B}} \sim S^{\langle p,p-1 \rangle} + S^{\langle p-1,p \rangle}$.          
\end{proof}

\begin{cor} \label{cor:nscondition}
If $S^{\lambda}$ is in $\mathcal{B}$ then $S^{\lambda}$ has Loewy length at most 4.  Moreover, a partition $\lambda$ in $\mathcal{B}$ is $p$-regular and $p$-restricted if and only if $S^{\lambda}$ has Loewy length 4.   
\end{cor}

We conclude this subsection by classifying the Specht modules in $\mathcal{B}$ having Loewy length 2.  This classification will aid in describing the socle of a Specht module in $\mathcal{B}$ that has Loewy length 3.

\begin{lem} \label{lem:spechtlength2help}
Let $\lambda$ be a $p$-regular partition in $\mathcal{B}$ and suppose there is $2 \leq i \leq p$ so that 
$S^{\lambda}\downarrow_{B_i} \cong S^{\tilde{\lambda}}$.

\begin{enumerate}
	\item  If $D^{\delta}$ is in the socle of $S^{\sigma_i(\lambda)}$ then 
		$\text{Hom}_{B_i}(D^{\delta}\downarrow_{B_i}, S^{\tilde{\lambda}}) \neq 0$.
	\item If $\lambda \in \Lambda_i$, $\sigma_i(\lambda)$ is $p$-regular, and $S^{\tilde{\lambda}}$ is irreducible then $S^{\lambda}$ and 			     $S^{\sigma_i(\lambda)}$ both have Loewy length at most 2. 
	\item If $S^{\lambda}$ has Loewy length 2 and $D^{\lambda} \downarrow_{B_i} =0$ then $S^{\tilde{\lambda}}$ is irreducible.  
\end{enumerate}
\end{lem}

\begin{proof}
We have a short exact sequence.
$$0 \rightarrow S^{\sigma_i(\lambda)} \rightarrow S^{\tilde{\lambda}}\uparrow^{\mathcal{B}} \rightarrow S^{\lambda} \rightarrow 0 $$  
(1)  Note that $D^{\delta}$ is in the socle of $S^{\tilde{\lambda}} \uparrow^{\mathcal{B}}$ if and only if 
$$0 \neq \text{Hom}_{\Sigma_{3p}} (D^{\delta},  S^{\tilde{\lambda}} \uparrow^{\mathcal{B}}) \cong 
\text{Hom}_{B_i}(D^{\delta}\downarrow_{B_i}, S^{\tilde{\lambda}}).$$  The result now follows since $\text{soc}\left(S^{\sigma_i(\lambda)} \right) \subseteq \text{soc}( S^{\tilde{\lambda}} \uparrow^{\mathcal{B}} )$.  We proceed to prove (2).  Since $S^{\tilde{\lambda}}$ is irreducible, $S^{\sigma_i(\lambda)}$ and $S^{\lambda}$ both have Loewy length at most 3 by Lemma \ref{lem:sinduce}(1).  Now, we have $S^{\sigma_i(\lambda)} \subseteq \text{rad} (S^{\tilde{\lambda}} \uparrow^{\mathcal{B}})$ since $S^{\sigma_i(\lambda)} \subseteq S^{\tilde{\lambda}} \uparrow^{\mathcal{B}}$ and $D^{\sigma_i(\lambda)} \downarrow_{B_i} = 0$. So $S^{\sigma_i(\lambda)}$ has Loewy length at most 2 since $S^{\tilde{\lambda}} \uparrow^{\mathcal{B}}$ has Loewy length 3.  Now suppose $D^{\delta}$ is in the socle of $S^{\lambda}$.  If $S^{\lambda}$ is irreducible then we are done.  If $S^{\lambda}$ is reducible then $D^{\delta} \downarrow_{B_i} = 0$ since $D^{\lambda}\downarrow_{B_i} \cong D^{\tilde{\lambda}}$ and $[S^{\tilde{\lambda}} : D^{\tilde{\lambda}}] = 1$. (i.e., the only composition factor of $S^{\lambda}$ that restricts to $D^{\tilde{\lambda}}$ is $D^{\lambda}$.)  So $D^{\delta}$ is not in the socle of $S^{\tilde{\lambda}} \uparrow^{\mathcal{B}}$ by part (1).  We conclude that $S^{\lambda}$ has Loewy length at most 2.  Finally, (3) follows since the Specht modules are indecomposable over odd characteristic.      
\end{proof}

The next lemma follows from the restrictions (1), (2), (3), and (4) from the proof of Theorem \ref{thm:gotit} and since $S^{\langle p,p-1 \rangle} \downarrow_{B_p} \cong S^{\langle p-1,p\rangle} \downarrow_{B_p}$.  
\begin{lem} \label{lem:killhead}
If $\lambda \notin \{\langle p,p-1\rangle, \langle p,p-1,p-2 \rangle \}$ is a partition in $\mathcal{B}$ that is $p$-regular, not $p$-restricted, and not a hook then there is $2 \leq i \leq p$ so that $S^{\lambda} \downarrow_{B_i}$ is non-zero and $D^{\lambda} \downarrow_{B_i} = 0$.   
\end{lem}

Consider the following sets.
	\begin{enumerate}
		\item [] $T_1 = \{\langle  i \rangle : 1 \leq i \leq p-1 \}$
		\item [] $T_2 = \{ \langle i,i,i \rangle : 2 \leq i \leq p\}$
		\item [] $T_3 = \{ \langle i,i \rangle : 1 \leq i \leq p \}$
		\item [] $T_4 = \{ \langle p, p-1 \rangle, \langle p-1,p \rangle, \langle 2,2,1 \rangle , \langle 2,1,1 \rangle \}$
	\end{enumerate}
Let $T  := T_1 \cup T_2 \cup T_3 \cup T_4$.  The next theorem classifies the Specht modules in $\mathcal{B}$ that have Loewy length 2.

\begin{thm} \label{thm:spechtlength2}
Let $\lambda$ be a partition in $\mathcal{B}$.  Then $S^{\lambda}$ has Loewy length 2 if and only if $\lambda \in T$.  
\end{thm}

\begin{proof}
Let $\lambda \in T$.  If $\lambda \in T_1 \cup T_2 \cup T_3$ then $\lambda$ is a hook partition different from both $(3p)$ and $(1^{3p})$.  So $S^{\lambda}$ has Loewy length 2 by Proposition \ref{prop:peel}.  Assume $\lambda \in T_4$.  If $\lambda = \langle p, p-1 \rangle$ then $\sigma_p(\lambda) = \langle p-1,p \rangle$.  Now $S^{\lambda} \downarrow_{B_p}$ is irreducible by Proposition \ref{prop:spechtsimpleinduce}(3). Hence, by Lemma \ref{lem:spechtlength2help}(2), $S^{\langle p,p-1 \rangle}$ and $S^{\langle p-1,p \rangle}$ both have Loewy length 2.  Now observe that $\langle p,p-1 \rangle ' = \langle 2,1,1 \rangle$ and 
$\langle p-1, p\rangle ' = \langle 2,2,1 \rangle$.  So $S^{\langle 2,2,1 \rangle}$   and $S^{\langle 2,1,1 \rangle}$ both have Loewy length 2 since $S^{\lambda} \otimes sgn \cong (S^{\lambda '})^{\ast}$.  Conversely, suppose the Specht module $S^{\lambda}$ in $\mathcal{B}$ has Loewy length 2.  By Proposition \ref{prop:3p}, if $\lambda$ is neither $p$-regular nor $p$-restricted then $\lambda = \langle i,i\rangle$ for some $1 \leq i \leq p$.  So $\lambda \in T_3$.  Note that a hook partition in $\mathcal{B}$ has 
$\langle 3^{p} \rangle$ notation of the form $\langle i \rangle$, $\langle i,i \rangle$, or $\langle i,i,i \rangle$ for $1 \leq i \leq p$.  So we may assume that $\lambda$ is not a hook.    First suppose that $\lambda$ is $p$-regular and not $p$-restricted.  If $\lambda = \langle p,p-1 \rangle$ then we are done.  Assume $\lambda \neq \langle p,p-1\rangle$.  Note that $\lambda \neq \langle p,p-1,p-2 \rangle$ by Proposition \ref{lem:help}.  So choose $2 \leq i \leq p$ such that $S^{\lambda} \downarrow_{B_i} \neq 0$ and $D^{\lambda} \downarrow_{B_i} = 0$ (see Lemma \ref{lem:killhead}).  Then $S^{\lambda} \downarrow_{B_i}$ is irreducible by \ref{lem:spechtlength2help}(3).  So $\lambda = \langle p-1,p \rangle$ by Proposition \ref{prop:spechtsimpleinduce}.  Finally, if $\lambda$ is $p$-singular and $p$-restricted then $\lambda ' \in \{\langle p,p-1 \rangle, \langle p-1,p \rangle \}$ since $\lambda$ is not a hook.  Hence $\lambda \in \{ \langle 2,1,1 \rangle, \langle 2,2,1\rangle\}$.                  
 We conclude that $\lambda \in T$.  
\end{proof}

\begin{cor} \label{cor:length3}
Let $\lambda$ be a partition in $\mathcal{B}$.  The Specht module $S^{\lambda}$ has Loewy length 3 if and only if $\lambda$ is not a hook and either $\lambda$ is $p$-regular, not $p$-restricted, and $\lambda \notin \{\langle p,p-1 \rangle, \langle p-1, p\rangle \}$ or 
$\lambda$ is $p$-singular, $p$-restricted, and $\lambda \notin \{\langle 2,1,1 \rangle, \langle 2,2,1 \rangle \}$.  
\end{cor}

\subsection{Some Loewy structure results}

Now that the upper bound on the Loewy lengths has been established, we turn our attention to describing some of the radical layers of the Specht modules in $\mathcal{B}$.  The discussion begins by first considering Specht modules in the block corresponding to partitions that are both $p$-regular and $p$-restricted.  Consider the following lemmas.  

\begin{lem} \label{lem:regristreg}
Assume $\lambda \neq \langle 1,2 \rangle$ is a partition in $\mathcal{B}$ that is $p$-regular and $p$-restricted.
	\begin{enumerate}
		\item If $\lambda \notin \{\langle 4,2,1 \rangle $, $\langle 3,3,2 \rangle\}$ then there is $2 \leq i \leq p$ such that $\lambda \in \Lambda_i$ and $\sigma_i(\lambda)$ is $p$-regular.
		\item If $\lambda = \langle 4,2,1\rangle $ then $S^{\langle 4,2,1\rangle} \downarrow_{B_4} \cong S^{\langle 3,2,1 \rangle}\downarrow_{B_4}$ and $S^{\langle 3,2,1 \rangle}$ has simple head isomorphic to $D^{\langle 4,2,1 \rangle}$.
	\end{enumerate}  
\end{lem}

\begin{proof}
Part (1) follows immediately from the proof of Lemma \ref{lem:remove}.  Next assume $\lambda = \langle 4,2,1 \rangle$.  Obviously, $S^{\langle 4,2,1 \rangle} \downarrow_{B_4} \cong S^{\langle 3,2,1\rangle}\downarrow_{B_4}$.  Now $S^{\langle 3,2,1 \rangle}$ has simple head isomorphic to $D^{\langle 4,2,1 \rangle}$ by Lemma \ref{lem:help} since $\langle3,2,1\rangle' = \langle p,p-1,p-2 \rangle$ and $\langle p,p-1,p-3 \rangle' = \langle 4,2,1 \rangle$.    
\end{proof}

\begin{lem} \label{lem:mullineux421}
We have the following.
$$ m(\langle p,p-1,p-3 \rangle) = 
\begin{cases}
\langle 3,5 \rangle &\text{if $p=5$}\\
\langle 6,5,3 \rangle & \text{if $p \geq 7$}
\end{cases}
$$
\end{lem}
\begin{proof}
 We begin by finding the Mullineux symbol for $m(\langle p,p-1,p-3 \rangle)$.  Start with the $\langle 3^p \rangle$ display $\Lambda = \langle p,p-1,p-3 \rangle$.  Let $m_1$ be the last bead of $\Lambda$, i.e., $m_1$ is the bead at position $4p$.  Now position $3p$ of $\Lambda$ is unoccupied, but there is a proper bead at a position before position $3p$.  So let $m_2$ be the bead of $\Lambda$ at position $3p-2$.  Position $2p-2$ of $\Lambda$ is occupied by an improper bead, so we stop.  We move $m_1$ to position $3p$ and $m_2$ to position $3p-3$ (the first unoccupied position of $\Lambda$).  This process removes one rim $p$-hook with 3 parts and then one box.  So $A_0 = p+1$ and $R_0 = 4$, and we obtain another abacus display $\Lambda_1$.  We repeat the process to the abacus display $\Lambda_1$.  Let $m_1$ be the bead at position $4p-1$ (i.e., the last bead of the display).  Position $3p-1$ of $\Lambda_1$ is unoccupied and $\Lambda_1$ has no proper bead before position $3p-1$, so we stop.  Move $m_1$ to position $3p-1$ of $\Lambda_1$.  This removes one rim $p$-hook with 3 parts.  Hence $A_1 = p$ and $R_1 = 3$, and we produce a new abacus display $\Lambda_2$.  Let $m_1$ be the bead of $\Lambda_2$ at position $4p-3$.  Now position $3p-3$ is occupied by an improper bead.  So we stop the process and move $m_1$ to the first unoccupied position of $\Lambda_2$ which is position $3p-2$.  This removes a hook of length 
$p-1$ with $3$ parts.  So $A_2 = p-1$ and $R_2 = 3$, and we produce a abacus display $\Lambda_3$ for the empty diagram.  Therefore, $\langle p,p-1,p-3 \rangle $ has Mullineux symbol 

$$G_p(\langle p,p-1,p-3 \rangle ) =   \left( \begin{matrix}p+1 & p & p-1 \\ 4 & 3 &3 \end{matrix} \right).$$
So $m(\langle p,p-1,p-3 \rangle )$ has Mullineux symbol
$$G_p(m(\langle p,p-1,p-3 \rangle) ) =   \left( \begin{matrix}p+1 & p & p-1 \\ p-2 & p-3 & p-3 \end{matrix} \right).$$

We first prove the result for $p = 5$.  It suffices to show that $m(\langle 5,4,2\rangle)$ and $\langle 3,5 \rangle$ have the same Mullineux symbol.  Let  $M$ be the $\langle 3^5 \rangle $ display $\langle 3,5 \rangle$.  Take $m_1$ to be the last bead of $M$, i.e., the bead of $M$ at position $23$.  Position $18$ of $M$ is unoccupied; however, $M$ has a proper bead at a position before position $18$.  So take $m_2$ to be the bead of $M$ at position $14$.  Now position $9$ of $M$ is occupied by an improper bead, so we stop.  Move $m_1$ and $m_2$ to positions $18$ and $13$ of $M$, respectively.  This removes one rim $5$-hook with 2 parts and then one box.  Hence $A_0 = 6$ and $R_0 = 3$, and we produce a new display $M_1$.  Let $m_1$ be the bead at position $20$ of $M_1$.  Now position $15$ of $M_1$ is unoccupied and $M_1$ has no proper bead before position $15$, so we stop.  Move $m_1$ to position $15$ of $M_1$.  This removes one rim $5$-hook with 2 parts.  So $A_1 = 5$ and $R_1 = 2$, and we obtain a new display $M_2$.  Let $m_1$ be the bead of $M_2$ at position $18$.  Now position $13$ of $M_2$ is occupied by an improper bead, so we stop.  Move $m_1$ to position $14$ (the first unoccupied position) of $M_2$.  This removes a hook of length $4$ with $2$ parts.  Thus $A_2 = 4$ and $R_2 = 2$, and we obtain a display for the empty diagram.  So
$$G_5(\langle 3,5 \rangle) = \left( \begin{matrix} 6 & 5 & 4 \\ 3 & 2 & 2  \end{matrix} \right) = G_5(m(\langle 5,4,2 \rangle)).$$  
Therefore, $m(\langle 5,4,2 \rangle) = \langle 3,5 \rangle$.  

We now show $m(\langle p,p-1,p-3 \rangle)  = \langle 6,5,3 \rangle$ for $p \geq 7$.  Let $X$ be the $\langle 3^p\rangle$ display $\langle 6,5,3 \rangle$.  Take $m_1$ to be the last bead of $X$, i.e., the bead of $X$ at position $3p+6$.  Now position $2p+6$ of $X$ is unoccupied; however, $X$ has a proper bead at an earlier position than $2p+6$.  So take 
$m_2$ to be the bead of $X$ at position $2p+4$.  Now position $p+4$ of $X$ is occupied by an improper bead, so we stop.  Move $m_1$ to position $2p+6$ and $m_2$ to position $2p+3$ (the first unoccupied position of $X$).  This removes one rim $p$-hook with $p-3$ parts and then one box.  Thus $A_0 = p+1$ and $R_0 = p-2$, and we obtain a new display $X_1$.  Now let $m_1$ be the bead of $X_1$ at position $3p+5$.  Position $2p+5$ of $X_1$ is unoccupied and $X_1$ does not have a proper bead before position $2p+5$, so we stop.  Move $m_1$ to position $2p+5$.  This removes one rim $p$-hook with $p-3$ parts.  So $A_1 = p$ and $R_1 = p-3$, and we obtain a new display $X_2$.  Let $m_1$ be the bead at position $3p+3$ of $X_2$.  Position $2p+3$ is occupied by an improper bead, so we stop.  Move $m_1$ to position $2p+4$ (the first unoccupied position of $X_2$).  This removes a hook of length $p-1$ with $p-3$ parts.  So $A_2 = p-1$ and $R_2 = p-3$.  The resulting display is for the empty diagram.  Therefore,
$$G_p(\langle 6,5,3\rangle) = \left(\begin{matrix} p+1 & p & p-1 \\ p-2 & p-3 & p-3 \end{matrix}  \right) = G_p(m(\langle p,p-1, p-3\rangle)).$$
So $m(\langle p,p-1, p-3 \rangle) = \langle 6,5,3\rangle$ for $p \geq 7$.                             
\end{proof}

Lemmas \ref{lem:regristreg} and \ref{lem:mullineux421} help us prove the following property of the Ext$^1$-quiver of $\mathcal{B}$.  The author wonders whether this property generalizes to Ext$^1$-quivers of arbitrary blocks of weight 3.

\begin{prop} \label{prop:mullextend}
If $\lambda$ is a $p$-regular and $p$-restricted partition in $\mathcal{B}$ then 
$$\text{Ext}^{1}_{\Sigma_{3p}} ( D^{\lambda} , D^{m(\lambda')} ) = 0.$$
\end{prop}

\begin{proof} First assume $\lambda \notin \{\langle1,2  \rangle,  \langle 3,3,2 \rangle\}$.  Let $\gamma : = m(\lambda')$, and assume by way of a contradiction that $\text{Ext}^1_{\Sigma_{3p}}(D^{\lambda}, D^{\gamma}) \neq 0$.  Choose $2 \leq i \leq p$ such that $\lambda \in \Lambda_i$.  Further, if $\lambda \neq \langle 4,2,1 \rangle$ we choose $i$ so that $\sigma_i(\lambda)$ is $p$-regular, as in Lemma \ref{lem:regristreg}(1). 
  Write $S^{\lambda}\downarrow_{B_i} \cong S^{\lambda_A}$.  Now $\text{Ext}^1_{\Sigma_{3p}}(D^{\lambda}, D^{\gamma}) \neq 0$ implies either:
	\begin{enumerate}
		\item $\gamma \in \Lambda_i$ and $\text{Ext}^1_{\Sigma_{3p-1}} (D^{\lambda_A} , D^{\gamma}\downarrow_{B_i}) \neq 0$ or 
		\item $\gamma \notin \Lambda_i$ and $[D^{\lambda_A}\uparrow^{\mathcal{B}}: D^{\gamma}] \neq 0$.  
	\end{enumerate}
We claim that (1) cannot happen.  Suppose (1) holds.  Then $D^{\gamma} \downarrow_{B_i}\cong D^{\gamma_C}$ where $C$ is the unique $(i-1)$-normal node for $\gamma$.  Now the partition $\lambda_A$ is both $p$-regular and $p$-restricted, so $S^{\lambda_A}$ has simple socle $D^{m_{3p-1}((\lambda_A)')}$.  By Proposition \ref{prop:mullk}, there exists a normal node $\tilde{B}$ for $\lambda'$ such that $\text{res} \tilde{B} = - \alpha$ and $\gamma_C = m(\lambda')_C = m_{3p-1}((\lambda)'_{\tilde{B}})$.  But then $\tilde{B}$ corresponds to a removable node for $\lambda$ of residue $\alpha$.  Therefore $\gamma_C = m_{3p-1}((\lambda_A)')$.  Observe that $\mathcal{P} \lambda_A = \mathcal{P}m_{3p-1}((\lambda_A))' = \mathcal{P}(\gamma_C)$.  Hence, $$\text{Ext}^1_{\Sigma_{3p-1}} (D^{\lambda_A} , D^{\gamma}\downarrow_{B_i}) \cong \text{Ext}^{1}_{\Sigma_{3p-1}}(D^{\lambda_A} , D^{\gamma_C}) = 0,$$
a contradiction.  So (2) must hold, that is, $\gamma \notin \Lambda_i$ and $[D^{\lambda_A}\uparrow^{\mathcal{B}} :D^{\gamma}] \neq 0$.  In particular, $D^{\gamma}\downarrow_{B_i} = 0$.  First consider $\lambda \neq \langle 4,2,1 \rangle$.  Since $\gamma \rhd \lambda \rhd \sigma_i(\lambda)$ and 
$\sigma_i(\lambda)$ is $p$-regular, we have $\text{Hom}_{\Sigma_{3p}}(S^{\sigma_i(\lambda)}, D^{\gamma}) = 0$.  We have a short exact sequence. 

\begin{equation}  \label{eq:SES11}
0 \rightarrow S^{\sigma_i(\lambda)} \rightarrow S^{\lambda_A}\uparrow^{\mathcal{B}} \rightarrow S^{\lambda} \rightarrow 0
\end{equation}  
Apply the functor $\text{Hom}_{\Sigma_{3p}} (-,D^{\gamma})$ to Equation \ref{eq:SES11}.  Now 
$D^{\gamma}\downarrow_{B_i} = 0$ and the Eckmann-Schapiro lemma yield $$\text{Ext}^{1}_{\Sigma_{3p}}(S^{\lambda}, D^{\gamma}) \cong \text{Hom}_{\Sigma_{3p}}(S^{\sigma_i(\lambda)}, D^{\gamma}) = 0.$$
We conclude that $\text{Hom}_{\Sigma_{3p}} (\text{rad} S^{\lambda}, D^{\gamma}) \cong \text{Ext}^1_{\Sigma_{3p}}(D^{\lambda}, D^{\gamma}) \neq 0$.  Now, by Theorem \ref{thm:RR}, $[S^{\lambda} : D^{\gamma}] \geq 2$, a contradiction.  Therefore, $\text{Ext}^1_{\Sigma_{3p}} (D^{\lambda}, D^{m(\lambda')}) = 0$ if $\lambda \neq \langle 4,2,1 \rangle$.  Next suppose that $\lambda = \langle 4,2,1 \rangle$.  Then $\sigma_4(\lambda) = \langle 3,2,1 \rangle$ and $\lambda '  = \langle p,p-1,p-3 \rangle$.  By Lemma \ref{lem:regristreg}, $S^{\langle 3,2,1 \rangle}$ has simple head isomorphic to $D^{\langle 4,2,1 \rangle}$.  Observe $\gamma = m(\langle p,p-1,p-3 \rangle)$.  By Lemma \ref{lem:mullineux421}, $\text{Hom}_{\Sigma_{3p}}(S^{\langle 3,2,1 \rangle}, D^{\gamma}) = 0$.  We have the following short exact sequence.

\begin{equation} \label{eq:SES22} 
0 \rightarrow S^{\langle 3,2,1 \rangle} \rightarrow S^{\lambda_A} \uparrow^{\mathcal{B}} \rightarrow S^{\lambda} \rightarrow 0  
\end{equation}       

\noindent Apply the functor $\text{Hom}_{\Sigma_{3p}} (-,D^{\gamma})$ to Equation \ref{eq:SES22}.  Now 
$D^{\gamma}\downarrow_{B_i} = 0$ and the Eckmann-Schapiro lemma yield $$\text{Ext}^{1}_{\Sigma_{3p}}(S^{\lambda}, D^{\gamma}) \cong \text{Hom}_{\Sigma_{3p}}(S^{\langle 3,2,1 \rangle}, D^{\gamma}) = 0.$$  We conclude that $\text{Hom}_{\Sigma_{3p}} (\text{rad} S^{\lambda}, D^{\gamma}) \cong \text{Ext}^1_{\Sigma_{3p}}(D^{\lambda}, D^{\gamma}) \neq 0$.  
Now, by Theorem \ref{thm:RR}, $[S^{\langle 4,2,1 \rangle} : D^{m(\langle 4,2,1\rangle')}] \geq 2$, a contradiction.  Therefore, it must be that $\text{Ext}^1_{\Sigma_{3p}} (D^{\lambda}, D^{m(\lambda')}) = 0$ for $\lambda \notin \{\langle1,2 \rangle , \langle3,3,2 \rangle\}$.  

If $\lambda= \langle 3,3,2 \rangle$ then $\lambda' = \langle p-2,p-1 \rangle$.  So the theorem holds for $\lambda'$.  Now, since the simple modules are self-dual,
$$ \text{Ext}^1_{\Sigma_{3p}} (D^{\langle 3,3,2 \rangle}, D^{m(\langle p-2,p-1 \rangle)}) \cong 
 \text{Ext}^1_{\Sigma_{3p}} ( D^{\langle p-2,p-1 \rangle }, D^{m(\langle 3,3,2 \rangle)})=0.$$ 
This same argument can be used to show the theorem holds for $\lambda = \langle 1,2 \rangle$ by considering $\lambda' = \langle p,p,p-1 \rangle$.          
\end{proof}

The next proposition together with Proposition \ref{prop:mullextend} will allow us to determine the composition factors of $S^{\lambda}$ from the $\text{Ext}^1$-quiver for $\mathcal{B}$, provided $\lambda$ is in $\mathcal{B}$ and is both $p$-regular and $p$-restricted.  See \cite{MR} for a construction of the $\text{Ext}^1$-quiver for $\mathcal{B}$.    

\begin{prop}[\cite{CT}, Theorem 6.1] \label{prop:ext2}
Let $\rho$ and $\sigma$ be $p$-regular weight 2 partitions of $n$ with $\sigma \geq \rho$.  Then 
$\text{Ext}^1_{\Sigma_n}(D^{\sigma}, D^{\rho}) \neq 0$ if and only if $[S^{\rho} : D^{\sigma}] \neq 0$ and $\partial \sigma$ and $\partial \rho$ have different parities.  
\end{prop}

Taken together, Propositions \ref{prop:w2}, \ref{prop:ext2}, and \ref{prop:james} imply Proposition \ref{prop:KS} holds for all weight 2 blocks of the symmetric group (over odd characteristic) with no restriction on the number of non-zero parts.  The goal now is to prove the restriction on the number of non-zero parts in Proposition \ref{prop:KS} can be removed for $p$-regular partitions in $\mathcal{B}$ that are also $p$-restricted.  First we introduce an assumption to better organize our work.

\noindent \textbf{Assumption.}
Say a $p$-regular partition $\lambda$ satisfies assumption $(*)$ if there exists $2 \leq i \leq p$ such that  
\begin{enumerate}
	\item $\left( S^{\lambda} \downarrow_{B_i} \right) \uparrow^{\mathcal{B}} \sim S^{\lambda} + S^{\sigma_i(\lambda)}$ such that $\lambda \in \Lambda_i$ and
	\item $\sigma_i(\lambda)$ is $p$-regular.  
\end{enumerate}

Let $R$ denote the set of all partitions in $\mathcal{B}$ that are both $p$-regular and $p$-restricted.  Define 
$$E : = \{\lambda \in R : \text{ $\lambda$ satisfies assumption $(*)$}\}.$$
Next we define the following sets:
 \begin{center} $E_1 = \{ \langle 5,2,1\rangle, \langle 5,4,1 \rangle , \langle 4,3,1 \rangle \}$ and
 $E_2  = \{ \langle 5,4,3 \rangle, \langle 4,4,2 \rangle, \langle 4,3,1 \rangle \}$. \end{center} 

\noindent By inspection, we have the following lemmas.   

\begin{lem}
Adopting the notation from above,
$$R - E = \{\langle1,2 \rangle, \langle 4,2,1\rangle, \langle 3,3,2 \rangle \}.$$
\end{lem}


\begin{lem} \label{lem:help2}
Let $\mu$ be a $p$-regular partition in $\mathcal{B}$.  The $\text{Ext}^1$-quiver for $\mathcal{B}$ gives the following.  
 	\begin{enumerate}
		\item If $\mu \ \rhd \langle 4,2,1\rangle $ and $\text{Ext}^1_{\Sigma_{3p}} (D^{\langle 4,2,1\rangle }, D^{\mu}) \neq 0$ then $\mu \in E_1$.
		\item If $\mu \ \rhd \langle 3,3,2 \rangle $ and $\text{Ext}^1_{\Sigma_{3p}} (D^{\langle 3,3,2 \rangle }, D^{\mu}) \neq 0$ then $\mu \in E_2$.
	\end{enumerate} 
\end{lem}

We now prove the second Loewy layer of $S^{\lambda}$ is determined by the Ext$^1$-quiver of $\mathcal{B}$ whenever $\lambda$ is $p$-regular and $p$-restricted.  

\begin{thm} \label{thm:2nd}
Let $\lambda$ and $\mu$ be $p$-regular partitions in the principal block of $F\Sigma_{3p}$ with $\lambda$ $p$-restricted.
If $\lambda$ does not strictly dominate $\mu$ then  
$$\text{Ext}^1_{\Sigma_{3p}} (D^{\lambda}, D^{\mu}) \cong \text{Hom}_{\Sigma_{3p}} (\text{rad} S^{\lambda}, D^{\lambda}).$$ 
\end{thm}

\begin{proof}
We may assume $\lambda$ has at least $p$ nonzero parts and $\text{Ext}^1_{\Sigma_{3p}}(D^{\lambda} , D^{\mu}) \neq 0$.  Assume $\lambda \neq \langle 1,2 \rangle$.  Choose $2 \leq i \leq p$ such that 
$\lambda \in \Lambda_i$, and let $A$ be the normal $(i-1)$-node of $\lambda$ such that $D^{\lambda}\downarrow_{B_i} \cong D^{\lambda_{A}}$.  Since $\text{Ext}^1_{\Sigma_{3p}}(D^{\lambda} , D^{\mu}) \neq 0$ we have either
	\begin{enumerate}
		\item $\mu \in \Lambda_i$ and $\text{Ext}^1_{\Sigma_{3p-1}} (D^{\lambda_A}, D^{\mu}\downarrow_{B_i}) \neq 0$ or 
		\item $\mu \notin \Lambda_i$ and $[D^{\lambda_A}\uparrow^{\mathcal{B}} : D^{\mu}] \neq 0$.
	\end{enumerate}  

\noindent 
Begin by assuming (1) holds.  Then $\text{Ext}^1_{\Sigma_{3p-1}}(D^{\lambda_A} , D^{\mu_B}) \neq 0$, 
where $B$ is the normal node for $\mu$ with $\text{res} B = \alpha$.  By Proposition \ref{prop:ext2}, we have $[S^{\lambda_A} : D^{\mu_B}] \neq 0$.  Then, by Lemma 2.13 of \cite{BK}, $[S^{\lambda} : D^{\mu}] \neq 0$.  Now 
$\text{Ext}^1_{\Sigma_{3p}}(D^{\lambda} , D^{\mu}) \neq 0$ implies $\mathcal{P} \lambda \neq \mathcal{P} \mu$.  Hence $D^{\mu}$ must lie in the second radical layer of $S^{\lambda}$ by Proposition \ref{prop:mullextend}.  So $\text{Hom}_{\Sigma_{3p}}(\text{rad} S^{\lambda}, D^{\mu}) \neq 0$.  The result now follows since $\text{Ext}^1_{\Sigma_{3p}}(D^{\lambda} , D^{\mu})$ is one-dimensional.   

Assume (2) holds.  Inducing $S^{\lambda_A}$ back up to $\mathcal{B}$ gives the following short exact sequence.
\begin{equation} \label{eq:SES2}
0 \rightarrow S^{\sigma_i(\lambda)} \rightarrow S^{\lambda_A} \uparrow^{\mathcal{B}} \rightarrow S^{\lambda} \rightarrow 0
\end{equation}
\noindent Applying the functor $\text{Hom}_{\Sigma_{3p}}(-,D^{\mu})$ to the short exact sequence in Equation 
\ref{eq:SES2} we obtain the following exact sequence.

$$\text{Hom}_{\Sigma_{3p}}(S^{\lambda_A} \uparrow^{\mathcal{B}} , D^{\mu})
\rightarrow \text{Hom}_{\Sigma_{3p}}(S^{\sigma_i(\lambda)}, D^{\mu}) \rightarrow
\text{Ext}^1_{\Sigma_{3p}} (S^{\lambda}, D^{\mu}) \rightarrow
\text{Ext}^1_{\Sigma_{3p}}(S^{\lambda_A} \uparrow^{\mathcal{B}} , D^{\mu})$$

\noindent Now $D^{\mu}\downarrow_{B_i} = 0 $ since $\mu \notin \Lambda_i$.  We conclude that
$$\text{Ext}_{\Sigma_{3p}}^1(S^{\lambda}, D^{\mu}) \cong \text{Hom}_{\Sigma_{3p}}(S^{\sigma_i(\lambda)} ,D^{\mu})$$
by the Eckmann-Schapiro lemma and Frobenius reciprocity.  

We claim that $\text{Hom}_{\Sigma_{3p}}(S^{\sigma_i(\lambda)}, D^{\mu}) = 0$ whenever $\lambda$ satisfies $(*)$.  If $\lambda$ satisfies $(*)$ then $S^{\sigma_i(\lambda)}$ has simple head isomorphic to $D^{\sigma_i(\lambda)}$.  Moreover,  $\lambda \rhd \sigma_i(\lambda)$ implies $\text{Hom}_{\Sigma_{3p}}(S^{\sigma_i(\lambda)}, D^{\mu}) = 0$ since $\lambda$ does not strictly dominate $\mu$.  Therefore, 
$\text{Ext}_{\Sigma_{3p}}^1(S^{\lambda}, D^{\mu}) = 0$.  So
$$\text{Hom}_{\Sigma_{3p}} (\text{rad}S^{\lambda},D^{\mu}) \cong \text{Ext}^1_{\Sigma_{3p}} (D^{\lambda}, D^{\mu}),$$ as desired.
 
Assume $\lambda$ does not satisfy assumption $(*)$.  First assume $\lambda = \langle 3,3,2 \rangle$.  Then $\lambda \in \Lambda_{2} \cap \Lambda_{3}$.  If $\mu \rhd \lambda$ and 
$\text{Ext}^1_{\Sigma_{3p}}(D^{\lambda}, D^{\mu}) \neq 0$ then $\mu \in E_2$ by Lemma \ref{lem:help2}.  It is easily verified that $E_2 \subset 
\Lambda_2 \cup \Lambda_3$.  Hence, for $\mu \in E_2$, $[S^{\lambda} : D^{\mu}] \neq 0 $ with $\mathcal{P}\lambda \neq \mathcal{P}\mu$.  Now assume $\lambda= \langle 4,2,1 \rangle$.  By Lemma \ref{lem:help2}, if $\mu \rhd \lambda$ and $\text{Ext}^1_{\Sigma_{3p}}(D^{\lambda}, D^{\mu}) \neq 0$ then $\mu \in E_1$.  Observe that $\lambda \in \Lambda_{4}$ and $\lambda \notin \Lambda_j$ for $j \neq 4$.  Also $E_1 \cap \Lambda_4 = \{\langle 5,4,1 \rangle \}$.  Hence 
$\text{Hom}_{\Sigma_{3p}}(\text{rad}S^{\lambda}, D^{\langle 5,4,1\rangle }) \cong \text{Ext}^1_{\Sigma_{3p}} (D^{\lambda}, D^{\langle 5,4,1 \rangle})$.  Take $\mu \in E_1 - \{\langle 5,4,1 \rangle \}$.  We have a short exact sequence.

\begin{equation}  \label{eq:332}
0 \rightarrow S^{\langle 3,2,1 \rangle } \rightarrow (S^{\lambda} \downarrow_{B_4})\uparrow^{\mathcal{B}} \rightarrow S^{\lambda} \rightarrow 0
\end{equation}

\noindent By Lemma \ref{lem:help} and Proposition \ref{prop:james}, $S^{\langle 3,2,1 \rangle}$ has the simple head isomorphic to $D^{m(\langle p-2 \rangle)} \cong D^{\langle 4,2,1 \rangle}$.  Hence $\text{Hom}_{\Sigma_{3p}} (S^{\langle 3,2,1 \rangle},D^{\mu}) = 0$ for any partition $\mu$ in the set $E_1-\{\langle 5,4,1 \rangle\}$.  The result now follows from Equation \ref{eq:332} and Frobenius reciprocity.                   

Now consider the case of $\lambda = \langle1,2 \rangle$.  Using the $\text{Ext}^1$-quiver for $\mathcal{B}$ we see: If $\langle1,2 \rangle$ does not strictly dominate a $p$-regular partition $\mu$ of $n = 3p$ and $\text{Ext}^1(D^{\langle 1,2\rangle}, D^{\mu}) \neq 0$ then $\mu = \langle 2,1\rangle$ or $\langle1,3\rangle$.  Observe that $\lambda \in \Lambda_1$ and $\lambda \notin \Lambda_i$ for all $2 \leq i \leq p$.  Assume $\mu = \langle1,3 \rangle$.  Then $\mu \in \Lambda_1$, so $[S^{\lambda} : D^{\mu}] \neq 0$.  The result follows since $\mathcal{P}\lambda \neq \mathcal{P} \mu$.
Now suppose $\mu = \langle 2,1 \rangle$.  We have a short exact sequence
\begin{equation} \label{eq:21}
 0 \rightarrow N \rightarrow \left(S^{\lambda} \downarrow_{B_1}\right)\uparrow^{\mathcal{B}} \rightarrow S^{\lambda} \rightarrow 0
 \end{equation}
where $N$ has a Specht filtration $N \sim S^{\langle p,2,1 \rangle} + S^{\langle p,p,2 \rangle}$ with $S^{\langle p,2,1 \rangle}$ at the top.   Hence, 
$\text{Hom}_{\Sigma_{3p}}(N , D^{\mu}) = 0$.   By Frobenius reciprocity and Equation \ref{eq:21}, 
$$\text{Ext}^1_{\Sigma_{3p}}(S^{\lambda}, D^{\mu}) \cong \text{Hom}_{\Sigma_{3p}}(N , D^{\mu}) = 0.$$                                           
\end{proof}



At this point the reader may wonder whether Proposition \ref{prop:KS} holds for all $p$-regular partitions in $\mathcal{B}$.  One is led to believe that it does since $\mathcal{B}$ is closely related to the blocks of $F\Sigma_{3p-1}$.  In Corollary \ref{cor:KS} we see Proposition \ref{prop:KS} does indeed hold.  
 
\begin{cor} \label{cor:KS}
Let $\lambda$ and $\mu$ be a $p$-regular partitions in $\mathcal{B}$.  If $\mu$ does not strictly dominate $\mu$ then
$$\text{Hom}_{\Sigma_{3p}}(\text{rad} S^{\lambda} , D^{\mu}) \cong \text{Ext}^1_{ \Sigma_{3p} } (D^{\lambda}, D^{\mu}).$$
\end{cor}  

\begin{proof}
It suffices to prove the theorem for a $p$-regular partition $\lambda$ in $\mathcal{B}$ that is $p$-regular, not $p$-restricted, a non-hook, and has at least $p$ non-zero parts.  Refer to the list of abacus displays from the proof of Lemma \ref{lem:nicerest}.  Of these displays the following have at least $p$ non-zero parts:
	\begin{enumerate}
		\item $\langle 1,p \rangle$ or
		\item  $\langle i,1 \rangle $, $2 \leq i \leq p$.
	\end{enumerate}    
	Assume $\lambda = \langle 1,p \rangle$.  Then $\lambda \in \Lambda_p$.  Write $S^{\widetilde{\lambda}} = S^{\lambda} \downarrow_{B_p}$ and $D^{\widetilde{\lambda}} = D^{\lambda} \downarrow_{B_p}$.  Suppose $\mu$ is a partition in $\mathcal{B}$ such that $\mu \rhd \lambda$ and 
	$\text{Ext}^{1}_{ \Sigma_{3p}} (D^{\lambda}, D^{\mu}) \neq 0$.  Then either
	\begin{enumerate}
		\item $\mu \in \Lambda_p$ and $\text{Ext}^{1}_{ \Sigma_{3p-1}} (D^{\widetilde{\lambda}}, D^{\widetilde{\mu}}) \neq 0$ where  $D^{\widetilde{\mu}} = D^{\mu} \downarrow_{B_p}$ or 
		\item $\mu \notin \Lambda_p$ and $[D^{\widetilde{\lambda}} \uparrow^{\mathcal{B}} : D^{\mu}] \neq 0$.  
	\end{enumerate}  
If (1) holds then $[S^{\widetilde{\lambda}} : D^{\widetilde{\mu}}] \neq 0$ by Proposition \ref{prop:ext2}.  Hence $[S^{\lambda} : D^{\mu}] \neq 0$ by Lemma 2.13 of \cite{BK}.  Notice that $S^{\lambda}$ has Loewy length 2 or 3 and $\mathcal{P} \lambda \neq \mathcal{P} \mu$.  Hence 
$\text{Hom}_{ \Sigma_{3p}} (\text{rad} S^{\lambda}, D^{\mu}) \neq 0$.  The result now follows since $\text{Ext}^{1}_{\Sigma_{3p}} (D^{\lambda}, D^{\mu})$ is one-dimensional.  Now assume (2) holds.  So $D^{\mu} \downarrow_{B_p} = 0$.  We have a short exact sequence.
\begin{equation} \label{eq:KS1}
0 \rightarrow S^{\langle 1,p-1 \rangle} \rightarrow (S^{\lambda} \downarrow_{B_p})\uparrow^{\mathcal{B}} \rightarrow S^{\langle 1,p \rangle} \rightarrow 0  
\end{equation}        
Since $\mu \neq \langle 1,p-1 \rangle$, $\text{Ext}^{1}_{ \Sigma_{3p}} (S^{\lambda}, D^{\mu}) \cong \text{Hom}_{ \Sigma_{3p}} (S^{\langle 1,p-1 \rangle}, D^{\mu}) = 0$.  We conclude that $\text{Hom}_{ \Sigma_{3p}} (\text{rad}S^{\langle 1,p \rangle}, D^{\mu}) \cong \text{Ext}^{1}_{\Sigma_{3p}} (D^{\langle 1,p \rangle},D^{\mu})$.  A similar argument can be used to prove the theorem for $\lambda = \langle i,1 \rangle \in \Lambda_{i}$ for $2 \leq i \leq p$.         
\end{proof}

\subsection{A result on the socles}
It is known that if $\lambda$ is $p$-restricted then $S^{\lambda}$ has simple socle isomorphic to $D^{m(\lambda ')}$.  So the Mullineux map gives the socle of $S^{\lambda}$ whenever $\lambda$ is $p$-restricted.  We are concerned with determining the socle of a Specht module in $\mathcal{B}$  that has Loewy length 3.  We claim that the socle is indeed the third radical layer.

\begin{thm} \label{thm:socles}
Let $\lambda$ be a $p$-regular partition in $\mathcal{B}$.  If $S^{\lambda}$ has Loewy length 3 then 
$\text{soc}\left( S^{\lambda} \right) \cong \text{rad}^2 S^{\lambda}$.  
\end{thm}

\begin{proof}
By Corollary \ref{cor:length3}, $\lambda$ must have $\langle 3^p \rangle$ notation of the form:
\begin{enumerate}
	\item $\langle i,p \rangle$, $1 \leq i \leq p-2$,
	\item $\langle i,j \rangle$, $1 \leq j < i \leq p$ and $(i,j) \neq (p,p-1)$,  or 
	\item $\langle p,p-1,p-2 \rangle$.
\end{enumerate}
By Lemma \ref{lem:help}, we know $\text{rad}^2(S^{\langle p,p-1,p-2 \rangle}) \cong \text{soc} (S^{\langle p,p-1,p-2 \rangle})$.  So we assume $\lambda \neq \langle p,p-1,p-2 \rangle$.  Suppose $\text{Hom}_{\Sigma_{3p}}(\text{rad} S^{\lambda}, D^{\delta}) \neq 0$.  We prove that $D^{\delta}$ is not in the socle of $S^{\lambda}$.  By Lemma \ref{lem:killhead} we may choose is $2 \leq i \leq p$ such that $S^{\lambda} \downarrow_{B_i} \neq 0$ and $D^{\lambda} \downarrow_{B_i} = 0$.  Furthermore, $\lambda = \sigma_i(\mu)$ for the unique $p$-regular partition $\mu$ in $\mathcal{B}$.  Write $S^{\mu} \downarrow_{B_i} \cong S^{\tilde{\mu}}$.  Now either $D^{\delta}\downarrow_{B_i} =0$ or $D^{\delta}\downarrow_{B_i}$ is simple.  If $D^{\delta}\downarrow_{B_i} =0$ then $\text{Hom}_{\Sigma_{B_i}}(D^{\delta}\downarrow_{B_i}, S^{\tilde{\mu}}) = 0$.  So $D^{\delta}$ is not in the socle of $S^{\lambda}$ by Lemma \ref{lem:spechtlength2help}(1).  Now suppose $D^{\delta} \downarrow_{B_i}$ is simple.  Since $D^{\lambda}\downarrow_{B_i} = 0$, $\text{Hom}_{\Sigma_{3p}}(\text{rad} S^{\lambda}, D^{\delta}) \neq 0$, and $S^{\tilde{\mu}}$ has simple head $D^{\tilde{\mu}}$, we have $D^{\delta} \downarrow_{B_i} \cong D^{\tilde{\mu}}$.  Next, by Proposition \ref{prop:spechtsimpleinduce},  $S^{\tilde{\mu}}$ is reducible since $\lambda$ is not in the set $ \{ \langle p,p-1 \rangle, \langle p-1, p \rangle \}$.  So $\text{Hom}_{B_i} (D^{\delta} \downarrow_{B_i}, S^{\tilde{\mu}}) = 0$ since $[S^{\tilde{\mu}} : D^{\tilde{\mu}}] = 1$.  Therefore $D^{\delta}$ is not in the socle of $S^{\lambda}$ by Lemma \ref{lem:spechtlength2help}(1).               
\end{proof}

\section{An upper bound on the composition lengths} \label{section:bound}
In the previous section we saw that the Loewy length of any Specht module in $\mathcal{B}$ is at most 4.  Moreover, $S^{\lambda}$ has Loewy length 4 if and only if $\lambda$ is both $p$-regular and $p$-restricted.  In this section we establish an upper bound of 14 on the composition lengths of the Specht modules in $\mathcal{B}$.  Furthermore, we produce a Specht module in the block that has 14 composition factors and we demonstrate that the only Specht modules that can have composition length 14 must correspond to a partition that is both $p$-regular and $p$-restricted.    

Theorem \ref{thm:2nd} describes the Loewy structure of the Specht module $S^{\lambda}$ in $\mathcal{B}$, provided $\lambda$ is $p$-regular and $p$-restricted.  For such $\lambda$, $\text{rad}S^{\lambda} / \text{rad}^2 S^{\lambda}$ is given by the $\text{Ext}^1$-quiver of $\mathcal{B}$.  To obtain $\text{rad}^2S^{\lambda} / \text{rad}^3 S^{\lambda}$ we apply Theorem \ref{thm:2nd} to $S^{\lambda'}$ and tensor $\text{rad}S^{\lambda'} / \text{rad}^2 S^{\lambda'}$ with the sign representation.  We now place an upper bound on the number of composition factors seen in the second Loewy layer of $S^{\lambda}$.  

Let $\lambda$ be a $p$-regular partition in $\mathcal{B}$, and let $E_{\lambda}$ be the set of all $p$-regular partitions $\mu$ in $\mathcal{B}$ such that $\mu \rhd \lambda$ and $\text{Ext}^{1}_{\Sigma_{3p}}(D^{\mu}, D^{\lambda})$ is non-zero.  We claim that $|E_{\lambda}| \leq 6$.  Assume $\lambda \in \Lambda_t$.  If $t =1$, the quiver of $B_1$ (see Figure 16 of \cite{MR}) tells us that the second Loewy layer $S^{\Theta_1(\lambda)}$ has at most one composition factor since $\Theta_1(\lambda)$ is $p$-regular.  In other words, $|E_{\lambda} \cap \Lambda_1| \leq 1$.  After analyzing the quivers (see Figures 5-8 of \cite{MR}) for the blocks $B_t$ ($2 \leq t \leq p$) we may conclude the following: Given a $p$-regular partition $\tilde{\lambda}$ in $B_t$ there are at most 3 $p$-regular partitions $\tilde{\mu}$ in $B_t$ such that $\tilde{\mu}$ strictly dominates $\tilde{\lambda}$ and $D^{\tilde{\mu}}$ extends $D^{\tilde{\lambda}}$.  So, by Proposition \ref{prop:ext2}, if $2 \leq t \leq p$ then the second Loewy layer of $S^{\Theta_t(\lambda)}$ has at most 3 composition factors since $\Theta_t(\lambda)$ is $p$-regular.  In other words, 
$|E_{\lambda} \cap \Lambda_t| \leq 3$.     

Now given $\mu \in E_{\lambda}$ either there exists $1 \leq t \leq p$ so that $\lambda , \mu \in \Lambda_t$ or $\mu \notin \Lambda_t$ whenever $\lambda \in \Lambda_t$.  In the former case, the quiver for the set $\Lambda_t$ tells us that $D^{\mu}$ extends $D^{\lambda}$.  In latter case, we see that restricting to any of the blocks $B_t$ of $F\Sigma_{3p-1}$ misses a non-split extension of $D^{\lambda}$ by $D^{\mu}$.  The next proposition, which is Lemmas 4.2, 4.3, and 4.4 of \cite{MR}, gives all pairs of $p$-regular partitions 
$(\mu ; \lambda)$ where $\mu > \lambda$ and a non-split extension of $D^{\mu}$ by $D^{\lambda}$ is missed by restriction.        

\begin{prop} \label{prop:completemiss}
The following is a complete list of pairs of $p$-regular partitions $(\mu ; \lambda)$ where $\mu > \lambda$, 
$\text{Ext}^{1}_{\Sigma_{3p}}(D^{\mu}, D^{\lambda}) \neq 0$, and $\mu \notin \Lambda_t$ whenever $\lambda \in \Lambda_t$.  

\begin{enumerate}
	\item For $2 \leq f \leq p$, $3 \leq g \leq p$ and $5 \leq h \leq p$
		\begin{enumerate}
			\item $(\langle f \rangle ; \langle f-1 \rangle)$, \ $(\langle 1 \rangle ; \langle p,1 \rangle)$
			\item $(\langle g, 1 \rangle ; \langle g-1,1 \rangle)$, \ $(\langle 2,1 \rangle ; \langle 1,2 \rangle)$
			\item $(\langle 1, 2 \rangle ; \langle p,2,1 \rangle)$, \ $(\langle h,2,1 \rangle ; \langle h-1,2,1 \rangle)$
			\item $(\langle 4,3,1 \rangle ; \langle 4,2,1 \rangle)$, \ $(\langle 4,3,2 \rangle ; \langle 4,3,1 \rangle)$
		\end{enumerate}
	\item For $5 \leq s \leq p$
	$$(\langle s,s-1,s-2 \rangle ; \langle s,s-1,s-3 \rangle) 
	\text{ and } (\langle s,s-1,1 \rangle ; \langle s,s-2,1 \rangle)$$
	\item For $3 \leq s \leq p$
	$$(\langle s-1,s \rangle ; \langle s-2,s \rangle) \text{ and } (\langle 1,p \rangle ; \langle 1, p-1 \rangle)$$
\end{enumerate}
\end{prop}

\begin{lem} \label{lem:farapart}
Assume $p \geq 7$.  If $\lambda = \langle i,j,k \rangle$ where $i > j > k \geq 2$ and $i-j \geq 2$ and $j-k \geq 2$ then $|E_{\lambda}| \leq 6$.    
\end{lem}

\begin{proof}
Observe that $\lambda \in \Lambda_k \cap \Lambda_j \cap \Lambda_i$ and $\lambda$ is not of the form (1), (2), nor (3) of Proposition \ref{prop:completemiss}.  The discussion above gives the estimate
$|E_{\lambda}| \leq 3 +3 +3 = 9$.  We lower this bound on $|E_{\lambda}|$ to 6 by applying the following strategy: Use the Ext$^1$-quiver of $\mathcal{B}$ to find three distinct partitions 
$\mu_1, \mu_2, \mu_3 \in E_{\lambda}$ so that $\mu_1 \in \Lambda_{k}\cap \Lambda_{i}$, $\mu_2 \in \Lambda_{k}\cap \Lambda_{j}$, and 
$\mu_3 \in \Lambda_{j}\cap \Lambda_{i}$.  If such partitions exist then $|E_{\lambda}| \leq 3 + 2 + 1 = 6$, for we have over counted these partitions.  If the triple $\mu_1, \mu_2, \mu_3$ does not exist, we provide the set $E_{\lambda}$. We consider four cases.

\vskip .1in
\noindent \textbf{Case 1.}  Assume $i - j \geq 3$ and $j - k \geq 3$.  First suppose that $i < p$, and consider the partitions 
$\mu_1 =\langle i, j+1, k \rangle  $, $\mu_2 = \langle i+1, j , k\rangle$, and $\mu_3 = \langle i,j,k+1 \rangle$.  Then 
$\mu_1 \in \Lambda_j \cap \Lambda_k$, $\mu_2 \in \Lambda_k \cap \Lambda_i$, and $\mu_3 \in \Lambda_j \cap \Lambda_i$ and $\mu_1, \mu_2, \mu_3 \in E_{\lambda}$ (by Corollary \ref{cor:CPext}).  If $i = p$ then $\mu_1 = \langle k , p \rangle$, $\mu_2 = \langle k,j \rangle$, and 
$\mu_3 = \langle p, j+1 , k \rangle$.  Therefore $|E_{\lambda}| \leq 6$.  

\vskip .1in
\noindent \textbf{Case 2.} Assume $i - j \geq 3$ and $j - k =2$.  Then $\lambda = \langle i, j , j -2 \rangle$ where $i - j \geq 3$.  If $i = p$ then $E_{\langle p, j, j-2 \rangle} = \{\langle j-2, j\rangle, \langle p, j+1, j-2 \rangle, \langle p, j, j-1\rangle, \langle p, j+1, j \rangle \}$.  If $ i < p$ then
$E_{\langle i,j,j-2 \rangle} = \{ \langle i+1,j,j-2 \rangle, \langle i , j+1, j-2\rangle, \langle i,j,j-1 \rangle, \langle i,j+1,j \rangle     \}$.  Hence, 
$|E_{\lambda}| \leq 6$.      

\vskip .1in
\noindent \textbf{Case 3.} Assume $i-j = 2$ and $j-k \geq 3$.  Then $\lambda = \langle i , i -2, k \rangle$ where $2 \leq k \leq i-5$.  First suppose that $i = p$.  Then $E_{\langle p , p-2, k \rangle} = \{\langle k, p  \rangle, \langle k, p-2\rangle,
\langle p,p-1,k \rangle, \langle p,p-2, k+1 \rangle \}$.  If $i < p$ then
$E_{\langle i,i-2, k\rangle } = \{\langle i+1, i-2, k \rangle, \langle i,i-1,k \rangle , \langle i,i-2, k+1 \rangle,  \langle i+1,i, k \rangle    \}$.  Therefore, 
$|E_{\lambda}| \leq 6$.

\vskip .1in
\noindent \textbf{Case 4.} Assume $i-j = 2$ and $j-k = 2$.  Then $\lambda = \langle i, i -2 , i -4 \rangle $ where $6 \leq i \leq p$.  First assume $i=p$.  Take $\mu_1 = \langle p-4, p\rangle$, $\mu_2 = \langle p-4, p -2 \rangle$, and $\mu_3 = \langle p-2 , p\rangle$.  Now if $i <p$ we have $E_{\langle i,i-2, i-4 \rangle} = \{ \langle i+1, i-2, i-4 \rangle, \langle i, i-1, i-4 \rangle , \langle i, i-2, i-3 \rangle , 
\langle i,i-1,i-2 \rangle, \langle i+1, i, i-2 \rangle, \langle i+1, i , i-4 \rangle  \}$.  Therefore, $|E_{\lambda}| \leq 6$.                      
\end{proof}

\begin{lem}
If $\lambda \neq (3p)$ is a $p$-regular partition $\mathcal{B}$ not of the form $\langle i,j,k \rangle$ where $i > j > k \geq 2$ and $i-j \geq 2, j-k \geq 2$ then $|E_{\lambda}| \leq 6$.  
\end{lem}

\begin{proof}
By Proposition \ref{prop:nremovables}(5), we know that $\lambda$ is in at most two of the sets $\Lambda_t$.   We consider two cases.

\vskip .1in
\noindent \textbf{Case 1.} Whenever $\text{Ext}^{1}_{\Sigma_{3p}} (D^{\mu}, D^{\lambda}) \neq 0$ and $\mu > \lambda$ there is $1 \leq t \leq p$ so that $\lambda$ and $\mu$ are both in $\Lambda_t$.  Then we have $\Theta_t(\lambda)$ and $\Theta_t(\mu)$ are both $p$-regular and $\Theta_t(\mu) > \Theta_t(\lambda)$.  First assume there is only one $t$ such that $\lambda \in \Lambda_t$.  If $t=1$ then $S^{\Theta_1(\lambda)}$ lies in the defect 1 block $B_1$ of $F\Sigma_{3p-1}$.  Hence the second Loewy layer of $S^{\Theta_1(\lambda)}$ has at most one composition factor.  We conclude that $|E_{\lambda}| \leq 1$.  If $t \geq 2$ then $S^{\Theta_t(\lambda)}$ lies in the defect 2 block $B_t$ of $F\Sigma_{3p-1}$.  The second Loewy layer of $S^{\Theta_t(\lambda)}$ has at most 3 composition factors.  We conclude that $|E_{\lambda}| \leq 3$.  Now suppose that 
$\lambda \in \Lambda_s \cap \Lambda_t$ for some $1 \leq s < t \leq p$.  If $s =1$ then, using the argument just given, $|E_{\lambda}| \leq 1 + 3 = 4$.  Finally, if $s \geq 2$ then $|E_{\lambda}| \leq 3 + 3 = 6$.

\vskip .1in
\noindent \textbf{Case 2.} There is a $p$-regular partition $\mu > \lambda$ in $\mathcal{B}$ so that 
$\text{Ext}^{1}_{\Sigma_{3p}} (D^{\mu}, D^{\lambda}) \neq 0$ and $\mu \notin \Lambda_t$ whenever 
$\lambda \in \Lambda_t$.  First assume $\lambda$ is in $\Lambda_t$ for the unique $1 \leq t \leq p$.  Then, by Proposition \ref{prop:completemiss}, 
$|E_{\lambda}| \leq 3+1 = 4$ if $\lambda \neq \langle 4,2,1\rangle$ and $|E_{\langle 4,2,1 \rangle}| \leq 3 + 2 = 5$.  Now suppose that $\lambda \in \Lambda_s \cap \Lambda_t$ for some $1 \leq s < t \leq p$.  If $s =1$ then $|E_{\lambda}| \leq 1 + 3 + 1 = 5$.  So assume that $s \geq 2$.  After analyzing the $\langle 3^p \rangle$ display for the partition $\lambda$ in the pair $(\mu ; \lambda)$ of Proposition \ref{prop:completemiss}, we need only consider $\lambda$ from the pairs $(\mu; \lambda)$ of (1) or (2) of Proposition \ref{prop:completemiss} and $\lambda \notin \{ \langle 1,p-1 \rangle, \langle 1,3\rangle \}$.  For such $\lambda$, $|E_{\lambda}| \leq 3 + 3 +1 = 7$.  We use the Ext$^1$-quiver of $\mathcal{B}$ to show that there is a partition $\mu \in E_{\lambda}$ such that 
$\mu \in \Lambda_{s} \cap \Lambda_{t}$.  If such a partition $\mu$ exists then $|E_{\lambda}| \leq 7-1 = 6$.     

First consider $\lambda = \langle s-2, s \rangle$ for $4 \leq s \leq p$.  Now $\lambda \in \Lambda_{s-2} \cap \Lambda_{s}$.  Using the Ext$^1$-quiver of $\mathcal{B}$ we see that $\mu = \langle s, s-2\rangle \in E_{\lambda}$.  Observe that $\mu \in \Lambda_{s-2} \cap \Lambda_{s}$. Next assume that $\lambda = \langle s, s-1, s-3\rangle$ where $5 \leq s \leq p$.  Then $\lambda \in \Lambda_{s-3} \cap \Lambda_{s-1}$.  If $s < p$ then $\mu = \langle s+1, s-1, s-3 \rangle \in E_{\lambda}$, and $\mu \in \Lambda_{s-3} \cap \Lambda_{s-1}$.  If $s = p$ then $\lambda = \langle p,p-1, p-3 \rangle$.  By the Ext$^1$-quiver, $\mu = \langle p-3, p-1 \rangle \in E_{\lambda}$.  Observe that $\mu \in \Lambda_{p-3} \cap \Lambda_{p-1}$.  Finally, consider $\lambda = \langle s, s-2, 1\rangle$ for $5 \leq s \leq p$.  Then $\lambda \in \Lambda_{s-2} \cap \Lambda_{s}$.  First suppose $s = 5$.  If $p = 5$ then $\mu = \langle 3, 5  \rangle \in  E_{\lambda}$ and $\mu, \lambda \in \Lambda_{3} \cap \Lambda_5$.  If $p \geq 7$ then $\mu = \langle 6,5,3 \rangle \in E_{\lambda}$ and $\mu, \lambda \in \Lambda_{3} \cap \Lambda_5$.  Finally, assume that $s \geq 6$.  The partition $\mu = \langle s,s-2, 2\rangle$ is in both $E_{\lambda}$ and $\Lambda_{s-2} \cap \Lambda_s$.                                                        
\end{proof}

\begin{thm}
If $\lambda \neq (3p)$ is a $p$-regular partition in $\mathcal{B}$ then $|E_{\lambda}| \leq 6$.  
\end{thm}

\begin{cor} \label{cor:compregrest}
If $\lambda$ in $\mathcal{B}$ is $p$-regular and $p$-restricted then the Specht module $S^{\lambda}$ has at most 14 composition factors.  
\end{cor}

Corollary \ref{cor:compregrest} gives an upper bound on the number of composition factors $S^{\lambda}$ has when $\lambda$ is $p$-regular and $p$-restricted.  For such a partition $\lambda$ in $\mathcal{B}$, we think of the corresponding Specht module $S^{\lambda}$  as ``large'' compared to other Spechts in $\mathcal{B}$ (in the sense of Loewy length).  We will see that any Specht module in $\mathcal{B}$ has at most 14 composition factors.  Specifically, if $\lambda$ in $\mathcal{B}$ is either $p$-regular and not $p$-restricted, $p$-singular and $p$-restricted, or neither $p$-regular nor $p$-restricted then $S^{\lambda}$ has strictly less than 14 composition factors.    We begin the discussion by providing a Specht module that has exactly 14 composition factors.     

\begin{prop} \label{prop:ex14}
There is at least one Specht module in $\mathcal{B}$ with composition length 14.  
\end{prop}

\begin{proof}
Consider $\lambda = \langle 5,3,1 \rangle$ which is $p$-regular and $p$-restricted.  Note that $\lambda' = \langle p, p-2, p-4 \rangle$, so $\lambda$ is self-conjugate only if $p=5$.  Observe that $\lambda$ has exactly $p$ parts.  Define the sets
\vskip .1in
\noindent \textbf{Case 1.} Assume $p=5$, so then $\lambda' = \lambda$.  From the Ext$^1$-quiver,  
$$E_{\lambda} = E_{\lambda '}= \{ \langle 3,5\rangle,\langle1,5\rangle, \langle1,3\rangle,\langle5,4,3\rangle,\langle5,4,1 \rangle, \langle 5,3,2\rangle \}.$$
By Theorem \ref{thm:2nd}, $D^{\mu}$ is in the second Loewy layer of $S^{\lambda}$ for each $\mu \in X$.  Now $\lambda ' =\lambda$ and Theorem \ref{thm:2nd} give 
$D^{\mu} \otimes sgn \cong D^{m(\mu)}$ is in the third Loewy layer of $S^{\lambda}$ for each $\mu \in X$.  Hence $S^{\lambda}$ has 14 composition factors.  

\vskip .1in
\noindent \textbf{Case 2.}  Assume $p \geq 7$.  From the Ext$^1$-quiver,
$$E_{\lambda}= \{\langle 6,5,3 \rangle ,\langle 6,5,1\rangle,\langle 6,3,1\rangle, \langle 5,4,3 \rangle , \langle 5,4,1 \rangle ,\langle 5,3,2 \rangle \}$$
and
\vskip .1in
\noindent $E_{\lambda'}=\{\langle p-2,p \rangle , \langle p-4,p \rangle ,\langle p-4,p-2 \rangle ,\langle p,p-1,p-2 \rangle , \langle p,p-1,p-4\rangle ,\langle p,p-2,p-3 \rangle \}$.  
\vskip .1in
\noindent The result now follows from Theorem \ref{thm:2nd} and Proposition \ref{prop:james}.  
\end{proof}

\begin{cor}
For $5 \leq i \leq p$, the Specht module $S^{\langle i , i -2, i-4 \rangle}$ has composition length 14.  
\end{cor}

\begin{proof}
The case $i = p$ is Proposition \ref{prop:ex14} since $\langle 5,3,1 \rangle' = \langle p, p-2, p-4 \rangle$.  Now assume that $i < p$.  We saw in the proof of Lemma \ref{lem:farapart}(case 4) that $E_{\langle i, i-2 ,i-4 \rangle}$  is
$$\{ \langle i+1, i-2, i-4 \rangle, \langle i, i-1, i-4 \rangle , \langle i, i-2, i-3 \rangle , 
\langle i,i-1,i-2 \rangle, \langle i+1, i, i-2 \rangle, \langle i+1, i , i-4 \rangle  \}.$$ 
Now 
$$\langle i,i-2, i-4\rangle ' = \langle p-i+5 , p-i+3, p-i+1 \rangle = \langle p-i+5, (p-i+5) -2, (p-i+5)-4\rangle.$$ 
Hence $|E_{\langle i, i -2, i-4 \rangle}| = 6 = |E_{\langle i,i-2, i-4 \rangle'}|$, so $S^{\langle i, i -2, i-4\rangle}$ has composition length 14 by Proposition \ref{prop:james} and Corollary \ref{cor:KS}.    
\end{proof}

Take $p =5$ in Proposition \ref{prop:ex14}.  The partition $(5,4,3,2,1)$ has $\langle 3^5\rangle$ display $\langle 5,3,1 \rangle$.  So we see that the Specht module $S^{(5,4,3,2,1)}$ has 14 composition factors.  We provide the explicit partitions which label the simple modules lying in the second radical layer of $S^{(5,4,3,2,1)}$. 
$$E_{(5,4,3,2,1)} = \{(8,6,1), (6^2,1^3), (6,4,2^2,1), (5^3), (5^2,3,1^2),(5,4^2,2)\}.$$
Computing the image of $X$ under the Mullineux map we obtain the labels for the third radical layer of $S^{(5,4,3,2,1)}$.
$$ m(E_{(5,4,3,2,1)}) = \{(5^2,4,1),(7,4,2^2),(6,5,2,1^2),(9,6),(8,5,2),(7,6,1^2) \}.$$
Finally, $S^{(5,4,3,2,1)}$ has simple head isomorphic to $D^{(5,4,3,2,1)}$ and simple socle isomorphic to $D^{m(5,4,3,2,1)} = D^{(7,5,2,1)}$.  So we know the radical series of $S^{(5,4,3,2,1)}$.  We now turn our attention to showing every Specht module in $\mathcal{B}$ corresponding to a partition that is $p$-regular and not $p$-restricted has composition length at most 14.  Furthermore, we show that any such Specht module in fact has composition length at most 13.     

The decomposition numbers in weight 3 blocks of the symmetric group are not well understood.  It is known that they are at most one; however, there is no nice closed form for the numbers.  The next proposition helps us relate the decomposition matrix for $\mathcal{B}$ to the decomposition matrix for $F\Sigma_{3p-1}$.         

\begin{prop}[\cite{JamesMathas}, Corollary 2.4] \label{prop:JM}
Suppose $w \leq 3$.  Then every decomposition number of the principal block of $\Sigma_{w p}$ is either zero or can be equated with an explicit decomposition number of $\Sigma_{wp - 1}$.  
\end{prop}

Suppose $S^{\lambda}$ and $D^{\mu}$ are in the principal block $\mathcal{B}$ with $[S^{\lambda} : D^{\mu}] \neq 0$.  Let $A_1, \ldots, A_k$ be the removable nodes for $\lambda$.  By Proposition \ref{prop:JM}, $\mu$ has at least one normal node of $p$-residue res$A_i$ for some $1 \leq i \leq k$.  It is known that a Specht module in defect 2 has at most 5 composition factors (cf. \cite{DE}, Theorem 5.6) and a Specht module in defect 1 has at most 2 composition factors.  In light of this we obtain an upper bound on the composition length of $S^{\lambda}$ (by Lemma 2.13 of \cite{BK}).  In this case 
$5k$ is an upper bound for the composition length of $S^{\lambda}$.  Now suppose $\lambda$ is $p$-regular with $\ell$ normal nodes.  If $\ell \geq 2$ then we can improve upon our upper bound, for we have over counted the composition factor $D^{\lambda}$ by $\ell -1$ (this counting argument works since $[S^{\lambda}:D^{\lambda}] = 1$).  Hence $5k - \ell + 1$ is an upper bound on the composition length of $S^{\lambda}$.                         

We restrict our attention to non-hook partitions $\lambda$ in $\mathcal{B}$ that are $p$-regular but not $p$-restricted.  Any such partition must have $\langle 3^p \rangle$ notation of the form:
	\begin{enumerate}
		\item $\langle i,p \rangle $ where $1\leq i \leq p-1$,
		\item $\langle i,j \rangle$ where $1 \leq j < i \leq p$, or
		\item $\langle p,p-1,p-2 \rangle$.  
	\end{enumerate}
By Lemma \ref{lem:help}, $S^{\langle p,p-1,p-2 \rangle}$ has at most 13 composition factors since the Specht module $S^{\langle p,p-1p-3 \rangle}$ has at most 14 composition factors.  So we only consider the partitions from (1) and (2) above.

Define $\Omega = \{ \langle i,p \rangle : 1 \leq i \leq p-1  \} \cup \{\langle i,j \rangle: 1 \leq j <i \leq p \}$ and define the collection 
$\Gamma =\{\langle i,j \rangle : 1 \leq j < i \leq p-1, \ i-j \geq 2\} \cup \{\langle p ,1 \rangle \} \subset \Omega$.  The discussion that follows Proposition \ref{prop:JM} together with Propositions \ref{prop:JM}, \ref{prop:nremovables} gives: If $\lambda$ is in $\Omega  - \Gamma$ then the Specht module $S^{\lambda}$ has at most 14 composition factors.  Now, consider $\lambda = \langle p,1 \rangle$.  The display $\langle p,1 \rangle$ has removable beads on runners $s = 1$, $2$, and $p$, so $S^{\langle p,1 \rangle}$ has a non-zero restriction to one defect 1 and two defect 2 blocks of $F\Sigma_{3p-1}$. Using our counting argument we see that $S^{\langle p,1 \rangle}$ has at most $2+5 + 5 - 1 + 1 = 12$ composition factors.  Hence we consider partitions that lie in 
$\Gamma - \{\langle p,1 \rangle \}$.

First fix $3 \leq i \leq p-1$.  The removable beads for $\langle i,1 \rangle $ lie on runners $s = 2, i, i+1$, and the normal bead lies on runner $s = i$.  So $\langle i,1 \rangle \in \Lambda_{i}$.  Furthermore, the partition $\Theta_m(\langle i,1 \rangle)$ is $p$-regular and not $p$-restricted ($m = 2, i, i+1$).  Now fix $i,j$ that satisfy $2 \leq j < i \leq p-1$ and $i-j \geq 2$.  The removable beads for $\langle i,j \rangle$ lie on runners $s = j, j+1, i , i+1$, and the normal beads for $\langle i,j \rangle$ lie on runners $s = j, i$.  So $\langle i,j \rangle \in \Lambda_{j} \cap \Lambda_{i}$.  Finally, the partition $\Theta_m(\langle i,j \rangle)$ is $p$-regular and not $p$-restricted ($m = j, j+1, i , i+1$).  By Propositions \ref{prop:w2}, \ref{prop:spechtsimpleinduce} we conclude the following.

\begin{prop} \label{prop:lastcase} The following statements hold.  
	\begin{enumerate} 
		\item For each $3 \leq i \leq p-1$ and $m = 2 , i , i+1$, the Specht module $S^{\Theta_m(\langle i,1 \rangle)}$ has Loewy length 2.
		\item For each $2 \leq j < i \leq p-1$ with $i-j \geq 2$ and $m = j, j+1, i,i+1$, the Specht module $S^{\Theta_m(\langle i,j \rangle)}$ has Loewy 2.
	\end{enumerate}
\end{prop}   

\noindent Proposition \ref{prop:lastcase} allows one to use the Ext$^1$-quivers given by Martin and Russell (see \cite{MR} Figures 5-8) to place the desired upper bound on the composition length of the Specht module $S^{\lambda}$ where $\lambda \in \Gamma$.  Henceforth, we assume $\Theta_s(\lambda)$ is expressed in $\langle 2, 3^{p-2}, 2^{s-2}, 3 \rangle$ notation (see Proposition \ref{prop:images}).    

Begin by fixing $3 \leq i \leq p-1$.  Using Proposition \ref{prop:images}(1) and Figures 5-8 of \cite{MR}, deduce that $\Theta_m(\langle i,1 \rangle)$ is connected to at most 3 vertices in the Ext$^1$-quiver of the block $B_m$ for each $m = 2, i, i+1$.  In other words, $S^{\Theta_m(\langle i,1 \rangle)}$ has at most 4 composition factors ($m = 2, i , i+1$). Hence, the Specht module $S^{\langle i,1 \rangle}$ has at most $3\cdot 4 -1 + 1 = 12$ composition factors.  Next fix $i,j$ that satisfy $2 \leq j < i \leq p-1$ and $i-j \geq 2$.  By the same argument given above the Specht module $S^{\langle i,j \rangle}$ has at most $4 \cdot 4 - 2+1 = 15$ composition factors.  We use the following argument to prove $S^{\lambda}$ has at most 13 composition factors where $\lambda = \langle i,j \rangle$.  Find all blocks $B_m$ ($2 \leq m \leq p$) of $F\Sigma_{3p-1}$ such that $S^{\lambda} \downarrow_{B_m} \neq 0$ and $D^{\lambda} \downarrow_{B_m} = 0$.  Choose the unique partition $\mu \in \mathcal{B}$, which is necessarily $p$-regular, such that $\mu > \lambda$ and  $S^{\mu}\downarrow_{B_m} \cong S^{\lambda} \downarrow_{B_m}$.  Then $D^{\mu} \downarrow_{B_m} \neq 0$.  Proposition \ref{prop:CP} will tell us that we have over counted the composition factor $D^{\mu}$ of $S^{\lambda}$.  From above, the removable beads for $\langle i,j \rangle$ lie on runners $s = j, j+1, i , i+1$ and $\langle i,j \rangle \in \Lambda_{j} \cap \Lambda_{i}$.  Observe that

$$S^{\langle i,j+1 \rangle} \downarrow_{B_{j+1}} \cong S^{\langle i,j \rangle}\downarrow_{B_{j+1}} \text{ and } S^{\langle i+1,j\rangle} \downarrow_{B_{i+1}} \cong S^{\langle i,j \rangle}\downarrow_{B_{i+1}}.$$
By Propositions \ref{prop:decnumb} and \ref{prop:CP}, $[S^{\langle i,j \rangle} : D^{\langle j+1,i \rangle}] = 1$ and $[S^{\langle i,j \rangle}: D^{\langle j,i+1 \rangle}] = 1$.  Now $\langle i,j+1 \rangle \in \Lambda_{i} \cap \Lambda_{j+1}$ and $\langle i+1,j \rangle \in \Lambda_{j} \cap \Lambda_{i+1}$ implies we have over counted the composition factors $D^{\langle i,j+1 \rangle }$ and $D^{\langle i+1,j \rangle}$ of $S^{\langle i,j \rangle}$.  Hence $S^{\langle i,j \rangle}$ has at most 13 composition factors.          

The argument  used to lower the composition length of $S^{\langle i,j \rangle}$ from 15 to 13 may be used to show that for any $\lambda \in \Omega$ the corresponding Specht module $S^{\lambda}$ has at most 13 composition factors.  We summarize the main result of the section in the following theorem.


\begin{thm} \label{thm:complengthbound}
Suppose the Specht module $S^{\lambda}$ lies in the principal block $\mathcal{B}$ of $F\Sigma_{3p}$.  Then $S^{\lambda}$ has at most 14 composition factors.  Furthermore, if $S^{\lambda}$ has composition length 14 then $\lambda$ is both $p$-regular and $p$-restricted.     
\end{thm}

\section{Conclusion and future research} \label{section:conclude}        
In this section we summarize the results of this paper, indicate what is still not known about the Specht modules in $\mathcal{B}$, and discuss Specht modules in other weight 3 blocks of the symmetric group.  We summarize the results of this paper in the following proposition.    

 \begin{prop}
 Suppose the Specht module $S^{\lambda}$ is in the principal block of $F\Sigma_{3p}$.
 	\begin{enumerate}
		\item The Loewy length of $S^{\lambda}$ is at most 4.  
				\begin{enumerate}
				\item $S^{\lambda}$ has Loewy length 1 if and only if $\lambda  \in \{(3p), (1^{3p}) \}$.
				\item $S^{\lambda}$ has Loewy length 2 if and only if $\lambda \in T$. (see Theorem \ref{thm:spechtlength2})  
				\item $S^{\lambda}$ has Loewy length 3 if and only if  $\lambda$ is not a hook and either $\lambda$ is 
				$p$-regular and not $p$-restricted or $p$-singular and $p$-restricted.
				\item$S^{\lambda}$ has Loewy length 4 if and only if $\lambda$ is $p$-regular and $p$-restricted. 
				\end{enumerate}
		\item If $\lambda$ is $p$-regular and $p$-restricted then $\text{Ext}^1_{\Sigma_{3p}} (D^{\lambda}, D^{m(\lambda ' )}) = 0$.  
		\item The module $S^{\lambda}$ has composition length at most 14.  Furthermore, if $S^{\lambda}$ has composition length 14 then $\lambda$ is $p$-regular and $p$-restricted.  
		\item Proposition \ref{prop:KS} holds for all $p$-regular partitions in $\mathcal{B}$.
		\item Suppose $S^{\mu}$ lies in a defect 3 block of $F\Sigma_{n}$.  If the partition $\mu$ is $p$-regular and $p$-restricted then $S^{\mu}$ has Loewy length 4.    
	\end{enumerate}	
 \end{prop}

We have discovered a number of results related to the structure of Specht modules in $\mathcal{B}$; however, we still do not have a general description of the socle of many of the modules.  Let $\lambda$ be a $p$-regular partition in $\mathcal{B}$. If $S^{\lambda}$ has Loewy length 2 then the socle of $S^{\lambda}$ is determined by the Ext$^1$-quiver of $\mathcal{B}$ (by Corollary \ref{cor:KS}).  If $S^{\lambda}$ has Loewy length 3 then $\text{soc}(S^{\lambda})$ is known in theory.  Indeed, one may pick a prime $p \geq 5$ and list the composition factors of $S^{\lambda}$.  In this case, by the bipartite nature of the Ext$^1$-quiver and Theorem \ref{thm:socles}, the socle of $S^{\lambda}$ is the direct sum of all composition factors $D^{\mu}$ of 
$S^{\lambda}$ such that $\mu \neq \lambda$ and $\mathcal{P} \mu = \mathcal{P} \lambda$.  So in theory we know the socles of the modules for each fixed prime $p \geq 5$, but we do not have a general answer.  For example, in the case of $p=5$ we found that every Specht module in the block has simple head and simple socle, except $S^{(9,4,2)}$ and $S^{(3^2,2^2,1^5)}$.  Explicitly, $\text{soc} (S^{(9,4,2)}) = D^{(15)} \oplus D^{(10,5)}$ and 
$S^{(3^2,2^2,1^5)}/\text{rad} S^{(3^2,2^2,1^5)} = D^{(4^3,3)} \oplus D^{(3^2,2^3,1^3)}$.  We are left with the question: Is there only one $p$-regular partition in $\mathcal{B}$ such that $S^{\lambda}$ does not have a simple socle?

One may wonder whether the results (2) - (4) extend to all defect 3 blocks of the symmetric group (over characteristic at least 5).  The author knows of no example where $D^{\lambda}$ extends $D^{m(\lambda')}$ in a defect 3 block of the symmetric group.  Perhaps this can be resolved using the algorithm, which constructs the Ext$^1$-quiver of a weight 3 block from the Ext$^1$-quiver of another weight 3 block, given by Martin and Tan in Proposition 4.25 of \cite{MT}.  

Finally, we saw that $S^{\lambda}$ in $\mathcal{B}$ has Loewy length 4 if and only if the corresponding partition is $p$-regular and $p$-restricted.  We proved the `only if' direction for any defect 3 block of $F\Sigma_n$.  It remains open whether the `if' direction holds.

\vskip .1in
\noindent \textbf{Acknowledgements.} I would like to thank my advisor, Dr. David Hemmer, for introducing me to this problem and for all of his guidance throughout this research.  I would also like to thank the referee for his or her helpful comments and suggestions.



\newpage

\end{document}